\documentclass[preprint, 10pt,a4paper]{elsarticle}

\usepackage{amssymb,amsmath,epsfig,amsthm}
\usepackage{graphicx}
\usepackage{caption, subcaption, comment}

\usepackage[colorlinks,allcolors=cyan!70!black]{hyperref}

\usepackage[all]{xy}
\usepackage{tikz-cd}
\usepackage{mathrsfs}
\usepackage{enumitem}

\makeatletter
\def\ps@pprintTitle{%
 \let\@oddhead\@empty
 \let\@evenhead\@empty
 \def\@oddfoot{\small{ }}%
 \let\@evenfoot\@oddfoot}
\makeatother

\theoremstyle{plain}
 \newtheorem{thm}{Theorem}[section]
 \newtheorem{prop}[thm]{Proposition}
 \newtheorem{lem}[thm]{Lemma}
 \newtheorem{cor}[thm]{Corollary}
\theoremstyle{definition}
 
 \newtheorem{dfn}{Definition}[section]
 \newtheorem{smt}{Statement}[section]
\theoremstyle{remark}
 \newtheorem{rem}{Remark}
 \numberwithin{equation}{section}

\renewcommand{\le}{\leqslant}
\renewcommand{\ge}{\geqslant}

\renewcommand\footnote[2][]{} 

    \usepackage{xcolor}
            \definecolor{mygray}{gray}{0.13}
            \colorlet{myblue}{blue!55!black}
            \definecolor{ramred}{RGB}{200,16,46}
            \definecolor{zow}{RGB}{50,200,150}

        \newtheorem*{sclaim}{Claim}
        
            \newcommand{\ds}[1]{\mathbb{#1}}

            \renewcommand{\hat}{\widehat} 
            \renewcommand{\phi}{\varphi}

                \newcommand{\rca}{{\sf RCA}_0}

                \newcommand{\join}{\oplus}
                \newcommand{\sig}[2]{\Sigma^{#1}_{#2}}
                \newcommand{\delt}[2]{\Delta^{#1}_{#2}}
                \newcommand{\pie}[2]{\Pi^{#1}_{#2}}
                \newcommand{\halts}{\downarrow}
                \newcommand{\zero}{\emptyset}
                \newcommand{\tred}{\le_\mathrm{T}}

                \newcommand{\wred}{\le_{\sf W}} 
                \newcommand{\swred}{\le_{\sf sW}} 
                \newcommand{\gwred}{\le_{\sf gW}} 
                \newcommand{\cred}{\le_{\sf c}} 
                \newcommand{\scred}{\le_{\sf sc}} 
                \newcommand{\wequiv}{\equiv_{\sf W}} 
                \newcommand{\swequiv}{\equiv_{\sf sW}} 
                \newcommand{\gwequiv}{\equiv_{\sf gW}} 
                \newcommand{\cequiv}{\equiv_{\sf c}} 
                \newcommand{\scequiv}{\equiv_{\sf sc}} 

                \newcommand{\cac}{{\sf CAC}}
                \newcommand{\caco}{\omega\text{-}{\sf CAC}}
                \newcommand{\scac}{{\sf SCAC}}
                \newcommand{\scaco}{\omega\text{-}{\sf SCAC}}

                \newcommand{\scacs}{\mathsf{SCAC^{small}}}
                \newcommand{\scacl}{\mathsf{SCAC^{large}}}
                \newcommand{\scact}{\mathsf{SCAC^{type}}}

                \newcommand{\lep}{\le_{\ds P}}
                \newcommand{\gep}{\ge_{\ds P}}

                \DeclareMathOperator{\lb}{lb}
                \DeclareMathOperator{\ran}{ran}

\begin{document}

\title{Variants of the chain-antichain principle in reverse mathematics}

\author[wssu]{Noah A.~Hughes}
\ead{hughesna@wssu.edu}

\address[wssu]{Department of Mathematics, Winston-Salem State University, 601 S. Martin Luther King Jr Dr, Winston-Salem, NC 27110, USA}

\vspace{18mm} \setcounter{page}{1} \thispagestyle{empty}

\begin{abstract}
Restricting the chain-antichain principle ($\cac$) to partially ordered sets which respect the natural ordering of the integers is a trivial distinction in the sense of classical reverse mathematics.
We utilize computability-theoretic reductions to formally establish distinguishing characteristics of $\cac$ and the aforementioned restriction to elaborate on the apparent differences obfuscated over $\rca$.
Stable versions of both principles are also analyzed in this way.
\end{abstract}


\maketitle


\section{Introduction}

This paper is a contribution to the study of computable combinatorics, reverse mathematics, and the broadening literature wherein computability-theoretic reductions are used to reveal fine differences that may be unapparent while working in the big five subsystems of second-order arithmetic.
Our goal is to analyze a natural weakening of the chain-antichain principle and its variants utilizing both frameworks to formally witness the core combinatorial differences in each statement. 

A partially ordered set (poset) $(P, \le_P)$ is a set $P \subseteq \ds N$ together with a homogeneous relation\footnote{
        A relation is \emph{homogeneous} (or an \emph{endorelation}) iff it is a relation on a set with itself, i.e., a subset of $X \times X$.
    }
$\le_P$ which is antisymmetric,\footnote{
        $(x \le_P y \land y \le_P x) \to (x = y)$,
    }
reflexive,\footnote{$(x \le_P x)$,}
and transitive.\footnote{$(x \le_P y \land y \le_P z) \to (x \le_P z)$.}
Below, when clear from context, $P$ and $\le_P$ will be used interchangeably to refer to the poset $(P,\le_P)$.
The undecorated symbol $\le$ is exclusively used to denote the standard integer ordering on $\ds N$.
We say $x,y \in P$ are \emph{$\le_P$-comparable} if $x \le_P y$ or $y \le_P x$, and otherwise call $x,y \in P$ \emph{$\le_P$-incomparable} and write $x \mid_P y$.
When $x \le_P y$, $y$ is called a \emph{$\le_P$-successor} of $x$ and $x$ is called a \emph{$\le_P$-predecessor} of $y$.
A \emph{chain} in $P$ is a set of pairwise $\le_P$-comparable elements and an \emph{antichain} is a set of pairwise $\le_P$-incomparable elements.

\begin{smt}[The chain-antichain principle]\label{cac-statement}
    If $(P,\le_P)$ is an infinite poset, then there is an infinite set $X \subseteq P$ such that $X$ is either a chain or antichain in $P$.
\end{smt}

\begin{dfn}
    We say that a poset $(P,\le_P)$ with $P \subseteq \ds N$ is \emph{$\omega$-ordered} if for all $x,y \in P$, $x \le_P y$ implies that $x \le y$.
\end{dfn}

Restricting the chain-antichain principle to $\omega$-ordered posets, also called \emph{naturally} ordered posets, provides a {\it prima facie} weakening of statement \ref{cac-statement} to analyze.
In the view of classical reverse mathematics, these two statements (the chain-antichain principle and its restriction to $\omega$-ordered posets) are equivalent (see Theorem \ref{caco-rm} below).
Using computability-theoretic analysis, however, the subtle combinatorial differences in these statements is revealed.
In particular, a degree-theoretic analysis of antichains in $\omega$-ordered posets without an infinite computable chain suffices to distinguish the statements.
In the course of this work, we also analyze the so-called {\it stable versions} of these two statements using computability-theoretic reductions between instance-solution problems.
Forcing arguments are used to distinguish the stable version of the chain-antichain principle from its restriction to $\omega$-ordered posets.

We assume familiarity with the rudiments of compatibility theory (see Soare \cite{soare_turing_comp_2016}, Chapter 2 of Dzhafarov and Mummert \cite{dzhafarov-mummert_rmprp_2022}) throughout. 
For a general introduction to the study of computable combinatorics in reverse mathematics we recommend Hirschfeldt \cite{hirschfeldt_slicing_2015} and  Dzhafarov and Mummert \cite{dzhafarov-mummert_rmprp_2022}.

\section{Reverse Mathematics}

Standard reverse mathematical analysis takes place over the axiom system $\rca$, a subsystem of second-order arithmetic with induction restricted to $\sig01$ formulas and set comprehension restricted to $\delt01$ formulas. 
When working over $\rca$, careful attention must be paid to nonstandard models. 
To maintain clarity, we use the symbol $\ds N$ to denote the first-order part of a model of $\rca$ and reserve $\omega$ to denote the standard natural numbers.
We call a model of $\rca$ in which $\ds N = \omega$ an \emph{$\omega$-model}.
For a comprehensive treatment of reverse mathematics over the base system $\rca$, we refer to Simpson \cite{Simpson_2009}.

Statement \ref{cac-statement} can be seen as an infinitary generalization of Dilworth's theorem \cite{dilworth_decomposition_1950}. 
Viewed over $\rca$, the chain-antichain principle is a simple consequence of Ramsey's theorem for pairs in two colors: color $\le_P$-comparable pairs red and $\le_P$-incomparable pairs blue; then a red homogeneous set is a chain and a blue homogenous set is an antichain. 
Cholak, Jockusch, and Soare \cite{cholak_strength_2001} asked (in Question 13.8) whether this implication reverses. 
Hirschfeldt and Shore \cite{hirschfeldt_combinatorial_2007} answered this open question in the negative (showing that Ramsey's theorem for pairs in two colors strictly implies the chain-antichain principle over $\rca$.)
{The original motivation for this question arose from the following computablitiy-theoretic result of Herrmann \cite{herrmann_infinite_2001}.}

\begin{thm}[Theorem 3.1 of \cite{herrmann_infinite_2001}]\label{herrmanns_result}
    There is an infinite computable poset $(H, \le_H)$ with $H \subseteq \omega$ that does not contain a $\delt02$ infinite chain or a $\delt02$ infinite antichain.  
\end{thm}

While our weakening of the chain-antichain principle is equivalent to the original statement over $\rca$, Herrmann's poset $(H,\le_H)$ can be used to separate the two principles via a computable reduction.
To begin, we prove the equivalence results for the original statements and their stable analogs.

\begin{thm}\label{caco-rm}
    Over $\rca$, the following are equivalent:
        \begin{enumerate}
            \item {The chain-antichain principle.}
            \item {The chain-antichain principle for $\omega$-ordered posets.}
        \end{enumerate}
\end{thm}

\begin{proof}
    We work in $\rca$.
    That item 1 implies item 2 is trivial: item 2 is a special case of item 1.
    To prove the reverse implication, fix an infinite poset $(P,\le_P)$ and set 
        \[
            \le_+ = \{(x,y) : x \le_P y \land x \le y\} \mbox{ and } \le_- = \{(y,x) : x \le_P y \land y \le x\}.
        \]
    Note that $\le_+$ and $\le_-$ are $\Delta^0_0$-definable over $P$ so $\rca$ proves these sets exist.
    Apply item 2 to $\le_+$ and obtain an infinite set $X$ which is either a chain or antichain in $\le_+$.
    Note if $X$ is a chain, then it is a chain under $\le_P$, and we are done.

    If instead $X$ is an antichain in $\le_+$, it may not be an antichain under $\le_P$.
    In this case, consider $\le_- \restriction (X \times X)$, and note that it is infinite and $\omega$-ordered.
    Again, apply item 2 and obtain an infinite set $Y \subseteq X$ which is either a chain or antichain with respect to $\le_-$. 

    We claim $Y$ is a chain or antichain in $\le_P$.
    Note that if $Y$ is a chain with respect to $\le_-$, then as before $Y$ is a chain with respect to $\le_P$.
    If $Y$ is instead an antichain in $\le_-$, then because $X$ is an antichain in $\le_+$, every pair in $Y$ is incomparable in both $\le_+$ and $\le_-$, and so also in $\le_P$.
    Hence $Y$ is an antichain in $\le_P$ and the proof is complete.
\end{proof}

\begin{rem}The equivalence over $\omega$-models of the chain-antichain principle and its restriction to $\omega$-ordered posets is implicit in Lemma 2.12 of Towsner \cite{towsner_constructing_2020}.\end{rem}

Analyzing {\it stable} versions of combinatorial principles is a fruitful technique established in Cholak, Jockusch and Slaman \cite{cholak_strength_2001}. 
Hirschfeldt and Shore \cite{hirschfeldt_combinatorial_2007} isolated the following notion of stability for posets.

\begin{dfn}
    For an infinite poset $(P, \le_P)$ we say an element $x \in P$ is 
        \begin{itemize}
            \item \emph{small} if $x \le_P y$ for all but finitely many $y \in P$
            \item \emph{large} if $y \le_P x$ for all but finitely many $y \in P$
            \item \emph{isolated} if $x \mid_P y$ for all but finitely many $y \in P$
        \end{itemize}
    We say $(P,\le_P)$ is \emph{stable} if all elements are either small or isolated, or all elements are either large or isolated.
    In the first case, we call $(P,\le_P)$ a stable poset of the \emph{small type}.
    In the second case, we call $(P,\le_P)$ a stable poset of the \emph{large type}.
\end{dfn}

\begin{rem}
    Any $\omega$-ordered stable poset is necessarily a stable poset of the small type.
\end{rem}

Again, when viewed over $\rca$, the chain-antichain principle for stable posets is equivalent to its further restriction to $\omega$-ordered stable posets.

\begin{cor}\label{scaco-rm}
    Over $\rca$, the following are equivalent:
        \begin{enumerate}
            \item {The chain-antichain principle restricted to stable posets.}
            \item {The chain-antichain principle restricted to $\omega$-ordered stable posets.}
        \end{enumerate}
\end{cor}

\begin{proof}
    Repeat the proof of theorem \ref{caco-rm} beginning with a stable poset $(P,\le_P)$. 
    If $P$ is of the small type, then $\le_+$ and $\le_-$ both define posets of the small type. If $P$ is of the large type, then $\le_+$ will render large elements $x \in P$ isolated and $\le_-$ will render them small. In either case, the posets $\le_+$ and $\le_-$ are both stable and $\omega$-ordered whence the proof goes through.
\end{proof}

\begin{rem}
    Corollary 3.6 and proposition 3.8 of Hirschfeldt and Shore \cite{hirschfeldt_combinatorial_2007} prove that the chain-antichain principle \emph{strictly} implies its restriction to stable posets. Consequently, in view of theorem \ref{caco-rm} and corollary \ref{scaco-rm}, we conclude that the chain-antichain principle for $\omega$-ordered posets {\it strictly} implies the chain-antichain principle for $\omega$-ordered stable posets over $\rca$.
\end{rem}



We conclude this section with a degree-theoretic analysis of the infinite antichains in $\omega$-ordered posets. Contrast Theorem \ref{caco-dtwosoln} with Herrmann's poset $(H, \le_H)$ from Theorem \ref{herrmanns_result}.

\begin{thm}\label{caco-dtwosoln}
    If $(P,\le_P)$ is an infinite computable $\omega$-ordered poset which contains no infinite computable chain, then $(P,\le_P)$ has an infinite $\delt02$ antichain.
\end{thm}

\begin{proof}
    Fix such a poset a $(P,\le_P)$.
    Notice the set $X$ of elements with no $\le_P$-successor is $\pie01$:
        \[
            X = \{ x \in \omega : \forall y (x < y \to x \not \le_P y)\}.
        \]
    We claim every integer $n$ must have a $\le_P$-successor $x_n \in X$.
    By way of contradiction, assume $n$ is the least integer with no $\le_P$-successor in $X$.
    We can then construct a computable chain via a greedy algorithm: let $y_0 = n$ and assume we have defined $Y_\ell = \{y_0 \le_P y_1 \le_P \cdots \le_P y_\ell\}$. 
    Since $y_0 \le_P y_\ell$, $y_\ell$ has no $\le_P$-successor in $X$.
    So we can choose the least integer $y \ne y_\ell$ such that $y_\ell \le_P y$ and set $Y_{\ell+1} = Y_\ell \cup \{y\}.$
    The set $Y = \bigcup_{\ell \in \omega} Y_\ell$ is a computable chain of $\le_P$, providing the desired contradiction.

    Set $x_0$ to be the least $\le_P$-successor of $0$ contained in $X$.
    Assume that $x_n$ has been defined and let $x_{n + 1}$ be the least element of $X$ $\le_P$-above $x_n + 1$.
    Notice $x_n \mid_P x_{n+1}$ because $x_n \not \le_P x_{n+1}$ since $x_n \in X$, and $x_{n+1} \not \le_P x_n$ since $x_{n+1} \ge x_n + 1 > x_n$.
    Let $Y = \{x_n : n \in \omega\}$ and note by construction that $Y \tred X$, so $Y$ is $\delt02$.
    Note $Y$ is an antichain because each element in $Y$ has no $\le_P$ successor. 
    Indeed, there can be no pair $x_m, x_n \in Y$ such that $x_m \le_P x_n$, so the proof is complete.
\end{proof}


\section{Computability-theoretic reductions}

\subsection{Definitions}

To formally distinguish the logical content of these various principles, we introduce the main computability-theoretic reductions used in standard reverse mathematics .
When discussing computability theoretic reductions, we restrict our attention to $\omega$-models.
That is, we work with $\omega$ and its subsets and not the more general set $\ds N$.
Note how the combinatorial principles herein have the following $\pie12$-gestalt:
    \[
        \forall X (\phi(X) \to \exists Y \psi(X,Y))
    \]
where $X$ is an infinite poset, $\phi$ is an arithmetical formula describing the properties of $X$ (e.g., stable or $\omega$-ordered), and $\psi$ is an arithmetical formula asserting that $Y$ is infinite and either a chain or antichain in $X$.
This pattern motivates the concept of \emph{instance-solution problems} (or simply \emph{problems} for short).
We formulate a mathematical principle $\sf T$ of $\pie12$-form as a problem with \emph{instances} being all $X$ that satisfy $\phi$ and \emph{solutions} (to $X$) being all $Y$ that, together with $X$, satisfy $\psi$.

More formally, we define problems as multifunctions between represented spaces when more generality is required to encode the objects referred to in $\sf T$.
This is the usual approach taken in computable analysis, but as all of the principles discussed in the course of this work concern only $\omega$ and its subsets, we do not utilize this generality.
We recommend Brattka, Gherardi and Pauly \cite{brattka-cherardi-pauly_weihrauch_comp_analysis_2021} for an introduction to this formulation and an overview of its applications in computable analysis.

\begin{dfn}
    Let $\cac$ denote the instance-solution problem with instances infinite posets $(P,\le_P)$ with $P \subseteq \omega$ and solutions $Y$ which are either an infinite chain or infinite antichain in $(P,\le_P)$.
\end{dfn}

\begin{dfn}
    The following instance-solution problems are all restrictions of $\cac$.
    \begin{itemize}[align=left]
        \item[$\scac$] is the problem $\cac$ restricted to stable instances.
        \item[$\caco$] is the problem $\cac$ restricted to $\omega$-ordered instances.
        \item[$\scaco$] is the problem $\cac$ restricted to $\omega$-ordered stable instances.
    \end{itemize}
\end{dfn}

\def\P{\mathsf{P}}
\def\Q{\mathsf{Q}}
Given two problems ${\sf P}$ and ${\sf Q}$ we say ${\sf P}$ is {\it reducible} to ${\sf Q}$, written $P \le_* Q$, if given any instance $X_{\sf P}$ of ${\sf P}$, there is a way to transform it (perhaps computably) into an instance $X_{\sf Q}$ of ${\sf Q}$, such that any solution $Y_Q$ to $X_Q$ can then be transformed into a solution $Y_{\sf P}$ of the original instance $X_{\sf P}$ of ${\sf P}$. We say ${\sf P}$ and ${\sf Q}$ are {\it equivalent}, written $\P \equiv_* \Q$, if $\P$ is reducible to $\Q$ and $\Q$ is reducible to $\P$. The four main reductions are defined next.

\begin{dfn}\label{reduc}
    Given two problems $\sf P$ and $\sf Q$, we say
    \begin{enumerate}
        \item {$\P$ is \emph{computably reducible} to $\Q$, written $\P \cred \Q$, if and only if given any instance $X_{\P}$ of $\P$, there is an instance $X_{\Q}$ of $\Q$ such that $X_{\Q} \tred X_{\P}$, and for any solution $Y_{\Q}$ of $X_{\Q}$, there is a solution $Y_{\P}$ of $X_{\P}$ such that $Y_{\P} \tred X_{\P} \join Y_{\Q}$.}

        \item {$\P$ is \emph{strongly computably reducible} to $\Q$, written $\P \scred \Q$, if and only if given any instance $X_{\P}$ of $\P$, there is an instance $X_{\Q}$ of $\Q$ such that $X_{\Q} \tred X_{\P}$, and for any solution $Y_{\Q}$ of $X_{\Q}$, there is a solution $Y_{\P}$ of $X_{\P}$ such that $Y_{\P} \tred Y_{\Q}$.}

        \item {$\P$ is \emph{Weihrauch reducible} to $\Q$, written $\P \wred \Q$, if and only if there are two fixed Turing functionals $\Phi$ and $\Psi$ such that, given any instance $X_{\P}$ of $\P$, the set defined by $\Phi^{X_{\P}}$ is an instance of $\Q$, and any solution $Y_{\Q}$ of this instance has that $\Psi^{X_{\P} \join Y_{\Q}}$ defines a solution to $X_{\P}$.}

        \item {$\P$ is \emph{strongly Weihrauch reducible} to $\Q$, written $\P \swred \Q$, if and only if there are two fixed Turing functionals $\Phi$ and $\Psi$ such that, given any instance $X_{\P}$ of $\P$, the set defined by $\Phi^{X_{\P}}$ is an instance of $\Q$, and any solution $Y_{\Q}$ of this instance has that $\Psi^{Y_{\Q}}$ defines a solution to $X_{\P}$.}
    \end{enumerate}
\end{dfn}

\begin{figure*}[t]
\centering
    \begin{tikzcd}
     & \swred \arrow[ld, Rightarrow] \arrow[rd, Rightarrow] & \\
        \wred \arrow[rd, Rightarrow] \arrow[d, Rightarrow] & & \scred \arrow[ld, Rightarrow] \\
        \gwred  & \cred &                          
    \end{tikzcd}
    \hskip2cm
    \begin{tikzcd}
            {\sf P} & \swred & {\sf Q}\\[-20pt]
            X_{{\sf P}} \arrow[rr, "\Phi"] \arrow[d, dotted] &     & X_{{\sf Q}} \arrow[d, dotted]   \\
            Y_{{\sf P}}  &     & Y_{{\sf P}} \arrow[ll, "\Psi"']
        \end{tikzcd}
\caption{The left diagram summarizes the implications between each sort of reduction. The right diagram provides a graphical representation of the reduction paradigm in the specific context of a strong Weihrauch reduction.}
    \label{ct_reductions}
\end{figure*}
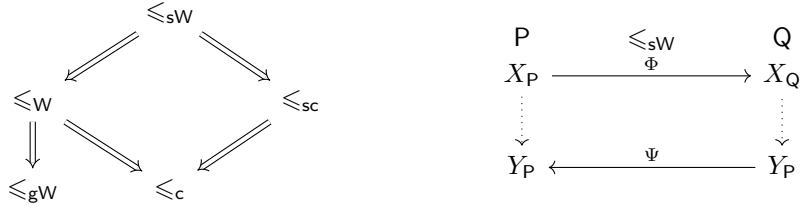

\def\Join{\bigoplus}
To incorporate the analysis of multiple uses of a mathematical principle in a proof, we utilize \emph{generalized Weihrauch reductions} via Hirschfeldt-Jockusch games (see Hirschfeldt and Jockusch \cite{hirschfeldt_notions_2016}, Definitions 4.1 and 4.3).
Define the \emph{$n$-fold join} of sets $X_0,\dots, X_n$ by
    \[
        \Join_{i \le n} X_i = \{n\} \join \{(i,k) : i \le n \land  k \in X_i \}.
    \]
When it is clear from context, we will write $X_0$ for $\Join_{i \le 0} X_i$ and $X_0 \join X_1$ for $\Join_{i \le 1} X_i$.

\begin{dfn}
    Given problems $\sf P$ and $\sf Q$, the \emph{reduction game} $G(\Q \to \P)$ is a two-player game ending when one player wins.
    The game is played as follows.
    
    For the first move, Player I begins by playing a $\P$-instance $X_0$.
    Player II responds by either playing an $X_0$-computable solution to $X_0$, in which case they win, or by playing a $\Q$-instance $Y_n \tred X_0$.
    If Player II cannot respond, Player I wins.

    On the $n$th move, for $n > 1$, Player I plays a solution $X_{n-1}$ to the $\Q$-instance $Y_{n-1}$.
    Player II responds by either playing a $(\Join_{i < n} X_i)$-computable solution to $X_0$, in which case they win, or by playing a $\Q$-instance $Y_n \tred (\Join_{i < n} X_i)$.

    If there is no move at which Player II wins, then Player I wins.
\end{dfn}

\begin{dfn}
    A \emph{computable strategy} for Player II in the reduction game $G(\Q \to \P)$ is a Turing functional which, when given the join of the first $n$ moves of Player I, outputs the $n$th move for Player II.
    Specifically, the strategy is a functional $\Phi$ such that, if $Z$ is the join of the first $n$ moves of Player I, then $\Phi^Z = V \join Y$ where $Y$ is the $n$th move of Player II, and $V = \{1\}$ if Player II wins with $Y$, or $V = \zero$ otherwise.
    The strategy is \emph{winning} if it allows Player II to win no matter what moves Player II makes.
\end{dfn}

\begin{dfn}
    We say $\P$ is \emph{Weihrauch reducible to $\Q$ in the generalized sense}, $\P \gwred \Q$, if there is a winning computable strategy for Player II in the game $G(\Q \to \P)$.
\end{dfn}

\subsection{Applications of the reverse mathematical proof}

Notice the proof of theorem \ref{caco-rm} can be carried out uniformly in the given poset $(P,\le_P)$.
Hence we obtain the following.

\begin{thm}\label{caco-gwequiv}
    $\cac \gwequiv \caco$.
\end{thm}

\begin{proof}
    To see $\cac \gwred \caco$, consider the following strategy for $G(\caco \to \cac)$.
    Given $(P,\le_P)$, an instance of $\cac$, compute $\le_+$ and play the infinite $\omega$-ordered poset $(P, \le_+)$.
    Now a solution to $(P, \le_+)$ is an infinite subset $X \subseteq P$.
    Given $X$ and $(P, \le_P)$, compute and play the infinite $\omega$-ordered poset $Q = (X, \le_- \restriction(X \times X))$.
    Any solution of $Q$ is identically a solution to $(P,\le_P)$ by the proof of Theorem \ref{caco-rm}.
    Thus, this strategy is winning.

    The reverse direction is trivial as every instance of $\caco$ is an instance of $\cac$.
    Hence $\caco \swred \cac$ and in particular $\caco \gwred \cac$.
\end{proof}

In light of corollary \ref{scaco-rm}, we may utilize a similar argument to obtain:

\begin{cor}\label{scaco-gwequiv}
     $\scac \gwequiv \scaco$.
\end{cor}

\begin{rem}
    Though the equivalence of $\cac$ and $\caco$ (formalized in second-order arithmetic) over $\omega$-models of $\rca$ is implicit in Towsner \cite{towsner_constructing_2020}, this result is also a consequence of theorem \ref{caco-gwequiv}. Moreover, by corollary \ref{scaco-gwequiv}, we conclude this is true of the stable versions as well. That is, $\scac$ and $\scaco$ (formalized in second-order arithmetic) are equivalent over $\omega$-models of $\rca$.
\end{rem}



\subsection{Computable reductions}

We turn our attention to separating these variants of the chain-antichain principle.
To begin, we show that $\cac$ is not equivalent to $\caco$ with respect to computable reductions.
This follows immediately from Theorems \ref{herrmanns_result} and \ref{caco-dtwosoln}.

\begin{thm}\label{caco-nceq}
    $\cac \not \cred \caco$.
\end{thm}

\begin{proof}
    Fix an infinite poset $(H,\le_H)$ satisfying the conclusion of Theorem \ref{herrmanns_result}.
    So $(H,\le_H)$ is a computable instance of $\cac$ with no $\delt02$ solutions.
    Suppose $(P,\le_P)$ is an infinite $\omega$-ordered poset computable from $(H,\le_H)$. 
    Note $(P,\le_P)$ is computable.
    If the instance $(P,\le_P)$ has a computable solution $X$, then $(P,\le_P) \oplus X$ is also computable and therefore cannot compute a solution to $(H,\le_H)$.
    If $(P,\le_P)$ has no computable solution, it must have a $\delt02$ solution $Y$.
    Again, $(P,\le_P) \oplus Y$ cannot compute a solution to $(H,\le_H)$. 
    In either case, we see that the computable reduction fails and the proof is complete.
\end{proof}


Interestingly, this phenomenon does not occur for the stable versions.
They are in fact equivalent with respect to a computable reduction.

\begin{thm}\label{scaco-ceq}
    $\scac \cequiv \scaco$.
\end{thm}

\begin{proof}
    It is immediate that $\scaco \cred \scac$ as every instance of $\scaco$ is an instance of $\scac$.
    To see that $\scac \cred \scaco$, fix an instance $(P,\le_P)$ of $\scac$ and assume without loss of generality that $P = \omega$.
    To proceed, we consider two cases according to whether $\le_P$ is of the small type or of the large type.
    While the two cases are entirely symmetric, the distinction is necessary for the following construction.

    Assume $\le_P$ is a stable poset of the small type.
    Define $\le_Q$ to be the largest subset of $\le_P$ respecting the natural order on $\omega$. So
        \(
            \le_Q\, = \{(m,n) : m \le n \land m \le_P n\}.
        \)
    Note $\le_Q$ is computable from $\le_P$ and an instance of $\scaco$.
    Apply $\scaco$ to obtain an infinite set $X$ which is a solution of $\le_Q$.
    If $X$ is a chain in $\le_Q$, then it is chain in $\le_P$, and therefore a solution.
    If, however, $X$ is an antichain, it may not be a solution of $\le_P$.
    Thus we $(X \join \le_P)$-computably thin $X$ so that in either case it is a solution of $\le_P$.

    Define the predicate $R(m,n)$ to be
        \[
            m \le_Q n \leftrightarrow m \le_P n.
        \]
    So $R(m,n)$ holds if $\le_Q$ preserves the correct relationship between $m$ and $n$ and $R(m,n)$ fails only if $m \mid_Q n$ but $m \nmid_P n$.
    With $X = \{x_0 < x_1 < x_2 < \cdots\}$, define
        \[
            X_0 = \{x_i \in X : R(x_0,x_i))\}.
        \]
    Because $\le_P$ is stable and of the small type $\le_Q$ misses at most finitely many comparabilities for each integer. So, since $X$ is infinite, $X_0$ is as well.

    Assume we have built $X_n = \{x^n_0 < x^n_1 < x^n_2 < \cdots\}$ where $x^n_m$ is the $m$th element of $X_n$ in increasing order. 
    Construct $X_{n + 1} \subseteq X_n$ by setting 
        \[
            X_{n + 1} = \{x^n_i : i \le n + 1\} \cup \{x^n_i : i > n + 1 \land R(x^n_{n + 1}, x^n_i)\}.
        \]
    Again, by the stability of $\le_P$, $X_{n + 1}$ is infinite.
    
    Let $Y = \lim_n X_n = \bigcap_{n \in \omega} X_n$.
    Note for all $x<y$ in $Y$ we have that $R(x,y)$ holds.
    Also, note $Y$ is an infinite subset of $X$. Indeed, for all $n$, 
    $$|Y| \ge \left|\bigcap_{i \le n} X_i \right| \ge n + 1$$ because in particular
        \[
            \{x_0 < x^0_1 < x^1_2 < \dots < x^n_{n+1}\} \subseteq \bigcap_{i = 0}^{n} X_i.
        \]
    Because $Y \subseteq X$, $X$ is a solution to $\le_Q$, and for all $x<y$ in $Y$, $R(x,y)$ holds, we conclude $Y$ is a solution to $\le_P$.
    
    To complete in the proof in the first case, we need show that $Y \tred (X \oplus \le_P)$. 
    To see this, note $n \in Y$ if and only if $n \in X_n$ as at the $n$th stage of the construction,
    and since each $X_n \tred (X \oplus \le_P)$, we have $Y \tred (X \oplus \le_P)$.

    The second case is symmetric: when $\le_P$ is of the large type, define $\le_Q$ by
        \[
            \le_Q \, = \{(m,n) : m \le n \land n \le_P m\}
        \]
    let $R(m,n)$ be the predicate
        \[
            m \le_Q n \leftrightarrow n \le_P m.
        \]
    Now, repeat the construction to complete the proof.
\end{proof}

\begin{rem}
    In the latter case, we call $(P, \le_Q)$ the {\it dual order} of $(P,\le_P)$. Put another way, we say $(P,\le_Q)$ is the {\it dual order} of $(P,\le_P)$ if for all $x,y \in P$, we have $$x \le_P y \leftrightarrow y \le_Q x.$$
\end{rem}

\subsection{Uniformity and stability}

The proof of theorem \ref{scaco-ceq} is very nearly a uniform construction.
The single bit of non-uniform information used in this proof was whether or not the stable poset $(P,\le_P)$ is of the small or large type.
With this data given, the above construction would uniformly find solutions to instances of $\scac$ using $\scaco$ in an effective manner.
This motivates the definition of the following problems.

\begin{dfn}
    The following instance-solution problems are variants of $\scac$.
    \begin{itemize}[align=left]
        \item[$\scacs$]{ is the problem $\scac$ restricted to instances of the small type.}
        \item[$\scacl$]{ is the problem $\scac$ restricted to instances of the large type.}
        \item[$\scact$]{ is the problem with instances of the form $(P,\le_P,T)$ where $(P,\le_P)$ is a stable poset and $T \in \{S,L\}$ and $T = S$ (or $L$) implies $(P,\le_P)$ is of the small (or large) type, and solutions which are either an infinite chain or an infinite antichain of $(P,\le_P)$.}
    \end{itemize}
\end{dfn}

\begin{prop}\label{cases}
    $\scacs \swequiv \scacl \swequiv \scact$
\end{prop}

\begin{proof}
    The dual order of any instance of $\scacs$ is an instance of $\scacl$, and vice versa. Similarly, any instance of $\scact$ can be uniformly transformed into an instance of $\scacs$ ($\scacl$) according to whether $T = S$ ($L$) or not. In all cases, the solution to the uniformly computed instance is identically a solution to the original instance.
\end{proof}

\begin{cor}\label{scaco-weq}
    $\scaco \wequiv \scacs \wequiv \scacl \wequiv \scact$
\end{cor}

\begin{proof}
    To show $\scaco \wequiv \scacs$, repeat the proof of the first case of Theorem \ref{scaco-ceq}. The other equivalencies follow from Proposition \ref{cases}.
\end{proof}

\begin{rem}
Because the proof of Theorem \ref{scaco-ceq} requires the original instance of $\scac$ in the computation of its solution, we do not obtain strong Weihrauch equivalences in Corollary \ref{scaco-weq}.
\end{rem}


\begin{thm}\label{poset-flip}
    $\scac \scequiv \scacs$.
\end{thm}

\begin{proof}
    Trivially, we obtain $\scacs \scred \scac$.
    To show $\scac \scred \scacs$, suppose $(P,\le_P)$ is an instance of $\scac$.
    If $(P,\le_P)$ is of the small type, we may witness this reduction with the identity functional.
    If $(P, \le_P)$ is of the large type, we compute from $(P,\le_P)$ the \emph{dual} partial order $(P,\le_P')$, i.e., $\le_P'$ is defined by
        \[
            x \le_P' y \mbox{ if and only if } y \le_P x.
        \]
    We claim $(P, \le_P')$ is a stable poset of the small type.
    Indeed, if $x$ is $\le_P$-isolated, then clearly $x$ is $\le_P'$-isolated.
    If $x$ is $\le_P$-large, then there is a $t$ such that if $y > t$, then $x \ge_P y$.
    This implies that for each $y > t$, $x \le_P' y$.
    So $x$ is $\le_P'$-small.
    This verifies the claim.

    Apply $\scacs$ to $(P,\le_P')$ to obtain an infinite set $X \subseteq P$ which is either a chain or antichain in $\le_P'$.
    If $X$ is an antichain in $\le_P'$, it is an antichain in $\le_P$.
    If $X$ is a chain in $\le_P'$, it a chain $\le_P$ in the opposite direction.
    In either case, $X$ is identically a solution of $(P, \le_P)$.
\end{proof}

\begin{cor}\label{scac-sceq}
    $\scac \scequiv \scacs \scequiv \scacl \scequiv \scact$.
\end{cor}

\begin{proof}
    It remains to show that $\scac \scred \scacl$ and $\scac \scred \scact$.
    For $\scact$, take a forward functional which appends $S$ to $(P,\le_P)$ if it is of the small type, else take one that appends $L$; the backward functional is the identity.
    For $\scacl$, proceed symmetrically to the proof of Theorem \ref{poset-flip}.
\end{proof}

\begin{rem}
    In the next section, we separate $\scac$ and $\scaco$ over both Weihrauch and strong computable reducibility. That is, we prove $\scac \not \wred \scaco$ and $\scac \not \scred \scaco$.
    Hence Theorem \ref{scaco-ceq} and Corollaries \ref{scaco-weq} and \ref{scac-sceq} are all optimal.
\end{rem}

\section{Separating the stable versions}

In this section we prove that Theorem $\ref{scaco-ceq}$ cannot be improved with respect to the other computability-theoretic reductions.
That is, we show $\scac \not \wred \scaco$ and that $\scac \not \scred \scaco$.
In each case, a forcing argument is used to build an appropriate instance of $\scac$ which diagonalizes any possible Weihrauch (strong computable) reduction to $\scaco$.
We assume familiarity with forcing in arithmetic and utilize the definition of the forcing relation given in \S 3.2 of Shore \cite{shore_lecture_forcing_2016}. We recommend section 2.3 of Hirschfeldt \cite{hirschfeldt_slicing_2015}, chapter 3 of Shore \cite{shore_lecture_forcing_2016}, and chapter 7 of Dzhafarov and Mummert \cite{dzhafarov-mummert_rmprp_2022} for the needed background on forcing. 

The forcing notion we use is adapted from section 4 of Astor, Dzhafarov, Solomon, and Suggs \cite{astor_uniform_2016}.
The conditions used in this case are finite posets $\pi$ on an initial segment of $\omega$ along with an {\it assignment function} $a$ that locks each element in $\pi$ to a stable limit behavior.
Let $\mathrm{FinPO}$ be the set of all finite partial orders on initial segments of $\omega$.
For each $\pi \in \mathrm{FinPO}$, let $|\pi|$ be the greatest $n$ such that $\pi$ orders $\omega \restriction n$ and use $\le_\pi$ to denote the ordering.
So $\pi = (\omega \restriction |\pi|, \le_\pi) = (|\pi|, \le_\pi)$.
We say a poset $(P,\le_P)$ {\it extends} and is {\it compatible} with $\pi$ if $\omega \restriction |\pi| \subseteq P$ and for all $x,y < |\pi|$ we have $x \le_P y$ if and only if $x \le_\pi y$.
Thus if $\rho \in \mathrm{FinPO}$ extends $\pi$ then $|\rho| \ge |\pi|$.
We canonically code each finite set $\pi \in \mathrm{FinPO}$ to ensure the map $\pi \mapsto |\pi|$ is effective.

\begin{dfn}
    Let $\mathbb{P}$ be the following notion of forcing. 
    A {\it condition} is a pair $p = (\pi^p, a^p)$ as follows:
        \begin{itemize}
            \item {$\pi^p \in \mathrm{FinPO}$ where $\mathrm{FinPO}$ is the set of all partial orders on initial segments of $\omega$.}
            \item {$a^p$ is a map $|\pi^p| \to \{S,L,I\} \times (\omega \restriction |\pi^p| + 1)$ such that either for all $n < |\pi^p|, a^p(n) \in \{S,I\} \times (\omega \restriction |\pi^p| + 1)$ or for all $n < |\pi^p|, a^p(n) \in \{L, I\} \times (\omega \restriction |\pi^p| + 1)$;}
            \item {if $a^p(x) = (S,t)$ and $y \le_{\pi^p} x$, then $y < t$ and $a^p(y) = (S,u)$ for some $u$;}
            \item {if $a^p(x) = (L,t)$ and $x \le_{\pi^p} y$, then $y < t$ and $a^p(y) = (L,u)$ for some $u$;}
            \item {if $a^p(x) = (S,t)$ or $a^p(x) = (L,t)$ and $x \mid_{\pi^p} y$, then $y < t$; and }
            \item {if $a^p(x) = (I,t)$ and $x \le_{\pi^p} y$ or $y \le_{\pi^p} x$, then $y < t$.}
        \end{itemize} 
    A condition $q$ \emph{extends} $p$, written $q \le_{\ds P} p$, if $\pi^q$ extends $\pi^p$ and $a^p \subseteq a^q$. Two conditions $p,q \in P$ are \emph{parallel} if $\pi^p = \pi^q$.
    For any element $x < |\pi^p|$, we call the number $t$ given by $a^p(x)$ the {\it stabilization point} of $x$.
    The valuation $V : \mathbb{P} \to \omega \times \omega$ is given by $V(p) = \pi^p$.
\end{dfn}

We use $G$ to denote the stable poset given by a filter $\mathfrak{F} \subseteq \ds P$ (so $G = (\omega, \le_G) = \bigcup_{p \in \mathfrak{F}} \pi^p$) and as a name for the generic poset in the forcing language for $\ds P$.

\subsection{Separation over Weihrauch reducibility}

The goal of this subsection is to prove that $\scac$ is \emph{not} Weihrauch reducible to $\scaco$.
We require two lemmas to begin.
We first prove that at any point in a construction using $\ds P$ we may change our mind regarding which sort of poset will be built. 
For example, if $p \in \ds P$ forces $G$ to be of the small type, then there is a parallel condition $q$ such that $\pi^p = \pi^q$, and $q$ forces $G$ to be of the large type.
We next prove that any condition $p$ can be extended to an infinite computable poset $G$ without a computable chain or computable antichain. 

\begin{lem}\label{mindchange}
    If $p \in \ds P$ is such that the range of $a^p$ is contained in $\{S,I\} \times \omega \restriction |\pi^p| + 1$, then there is a parallel condition $q \in \ds P$ such that for all $x < |\pi^p|$
        \[
            (a^p(x) = (S,t) \to a^q(x) = (I,|\pi^p|)) \land (a^p(x) = (I,t) \to a^q(x) = (L,|\pi^p|)).
        \]
\end{lem}

\begin{proof}
    Fix $p$ and let $q = (\pi^p, a^q)$ with $a^q$ defined as follows: for each $x < |\pi^p|$ if $a^p(x) = (S,t)$ for some $t$ then $a^q(x) = (I, |\pi^p|)$, and if $a^p(x) = (I,t)$ then $a^q(x) = (L,|\pi^p|)$.
    We need to verify $q \in \mathbb{P}$. 
    It suffices to show that $a^q$ satisfies the required properties.
    This will complete the proof because $p$ and $q$ are clearly parallel.

    Fix an arbitrary $x < |\pi^p|$.
    Note $a^q(x) \ne (S,t)$ for any $t$.
    Suppose $a^q(x) = (I,t)$ for some $t$.
    Then $t = |\pi^p|$ by definition.
    So if $y \le_{\pi^p} x$ or $x \le_{\pi^p} y$ we have $y < |\pi^p|$ because $y \in \pi^p$, satisfying the two required properties.
    
    Suppose $a^q(x) = (L,t)$ for some $t$. 
    Again $t = |\pi^p|$ by definition.
    Moreover, $a^p(x) = (I,u)$ for some $u$.
    If for some $y$, we have $x \mid_{\pi^p} y$ then $y < t = |\pi^p|$ as needed.
    If for some $y$, we have $x \le_{\pi^p} y$, we claim $a^p(y) = (I,v)$ for some $v$ and therefore $a^q(y) = (L, |\pi^p|)$ satisfying the needed property.
    To verify the claim, suppose not: so by hypothesis $a^p(y) = (S,v)$ for some $v$.
    Since $x \le_{\pi^p} y$, we have $x < v$ and $a^p(x) = (S,w)$ for some $w$.
    But $a^p(x) = (I,u)$, a contradiction.
\end{proof}

\begin{lem}\label{nocompsoln}
    For any condition $p \in \ds P$, there is an infinite computable stable poset $G$ extending $\pi^p$ that does not contain an infinite computable chain or an infinite computable antichain.
\end{lem}
\begin{proof}
    Given $p$, we computably construct a sequence of $\ds P$ conditions $p =  p_0 \gep p_1 \gep \cdots$ such that $G = \bigcup_{n \in \omega} \pi^{p_n}$.
    To ensure $G$ is as desired we meet the requirements
        \begin{quote}
            $\mathscr{R}_e:$ If $\phi_e$ is the characteristic function of an infinite set, then 
                \[
                    \exists x,y,z,s \left((\phi_{e,s}(x) = \phi_{e,s}(y) = \phi_{e,s}(z) = 1) \land (x \mid_G y \land y <_G z) \right).
                \]
        \end{quote}
    Divide the stages of the construction up so that each requirement is addressed infinitely often.
    Suppose inductively we are at stage $t$ and given condition $p_n$.
    Assume $t$ is the $m$th stage dedicated to $\mathscr{R}_e$.
    If there are $x,y,z,s < |\pi^{p_n}|$ that witness $\mathscr{R}_e$ is satisfied, do nothing.
    Else, search for a condition $q$ with $|\pi^q| = |\pi^{p_n}| + m$ and $x,y,z,s < |\pi^q|$ such that  
        \[
            (\phi_{e,s}(x) = \phi_{e,s}(y) = \phi_{e,s}(z) = 1) \land (x \mid_{\pi^q} y \land y <_{\pi^q} z) .
        \]
    If such a condition exists, set $p_{n+1} = q$.
    Otherwise, set $p_{n+1} = p_n$ and proceed to the next stage.
    This completes the construction.

    Since only finitely many conditions $q$ are considered at each stage, the sequence $\langle p_n \rangle_{n \in \omega}$ and hence $G$ are computable.
    Clearly, every $\mathscr{R}_e$ is either vacuously satisfied or $m$ will grow sufficiently large to find witnesses $x,y,z$ and $s$ with which to satisfy it at some stage.
    Thus, if $\phi_e$ is the characteristic function of an infinite set, $\mathscr{R}_e$ guarantees it is not a chain or antichain in $G$.
\end{proof}


To separate $\scac$ and $\scaco$ with respect to Weihrauch reducibility, we first show that $\scac \not \swred \scaco$ and then appeal to a general fact about the proof to obtain the separation as a corrolary.
To show $\scac \not \swred \scaco$, we proceed by contradiction assuming there are fixed Turing functionals, $\Phi$ and $\Psi$, witnessing $\scac \swred \scaco$ and construct a condition $r$ that yields a ``Seetapun configuration:'' a collection of finite sets $F_0,\dots,F_n,E$ where at least one set must extend to a solution $Y$ of $\Phi^G$ such that $\Psi^{Y}$ cannot be a solution of $G$. 
This yields the needed contradiction.

\begin{thm}\label{scaco-no-swred}
            $\scac \not \swred \scaco$.
        \end{thm}

        \begin{proof}
            By way of contradiction, assume the functionals $\Phi$ and $\Psi$ witness $\scac \swred \scaco$.
            Thus given any infinite stable poset $\le_P$, we have that $\Phi^{\le_P}$ is an infinite $\omega$-ordered stable poset and for any infinite chain or infinite antichain $Y$ in $\Phi^{\le_P}$, we have that $\Psi^Y$ is either an infinite chain or infinite antichain in $\le_P$.
            To derive our final contradiction, we construct an infinite computable stable poset $G$ such that $\Phi^G$ contains either an infinite computable chain or an infinite antichain $Y$ where $\Psi^Y$ is not a solution to $G$.
            This is done by constructing a suitable condition $r$ and applying Lemma \ref{nocompsoln} to obtain $G$.

            Construct $r$ as follows. 
            Begin with any condition $p$ such that $|\pi^p|$ is sufficiently large to reveal a finite chain $F_0$ in $\Phi^{\pi^p}$ for which a pair $x_0,y_0 < |\pi^p|$ have that $\Psi^{F_0}(x_0)\halts=\Psi^{F_0}(y_0)\halts = 1$.
            Such a $p$ must exist, for if not, every condition $p$ forces any infinite chain $X$ in $\Phi^G$ to have that the set defined by $\Psi^X$ is finite, contradicting that $\Psi^X$ must be a solution to $G$.
            Therefore, $\Phi^G$ must not contain any infinite chains. 
            So, cofinitely many elements in $\Phi^G$ are isolated and thus $\Phi^G$ contains an infinite computable antichain $Y$.
            This again contradicts the supposition about $\Phi$ and $\Psi$ because, by Lemma \ref{nocompsoln}, $G$ has no computable solutions and so $\Psi^Y$ is not a solution to $G$.

            As the only property of $p$ needed to find $F_0$ is that $|\pi^p|$ sufficiently large, we can further require that $a^p(x_0) = (S,|\pi^p|)$ and $a^p(y_0) = (I, |\pi^p|)$ if $x_0 <_{\pi^p} y_0$, or $a^p(x_0) = (I, |\pi^p|)$ and $a^p(y_0) = (S,|\pi^p|)$ otherwise.
            Note then that $p$ forces $\max F_0$ isolated in $\Phi^G$ because if not, there is a condition $q \le_{\ds P}$ and a generic $G$, obtained from Lemma \ref{nocompsoln}, compatible with $q$ wherein $\max F_0$ is small in $\Phi^G$.
            So there must be an infinite chain $F_0 \cup X$ (with $F_0 < X$) in $\Phi^G$ (since if $\Phi^G$ has no infinite chain, it must contain a computable solution.)
            But by construction, $\Psi^{F_0 \cup X}$ will contain both a small and an isolated element and thus not be a solution to $G$, again contradicting our assumption on $\Psi$.
            
            Assume we are at condition $q_{n-1}$ and have found finite sets $F_0, \dots, F_{n-1}$ such that for each $i \le n - 1$
                \begin{itemize}
                    \item {$F_i$ is a chain in $\Phi^{\pi^{q_{n-1}}}$ and $q_{n-1}$ forces $\max F_i$ is isolated in $\Phi^G$;}
                    \item {$\min F_i \mid_{\Phi^{\pi^{q_{n-1}}}} \max F_j$ for all $j < i$ (to ensure the maxima form an antichain);}
                    \item {there is a pair $x_i,y_i$ such that $\min\{x_i,y_i\} > \max\{x_j,y_j : j < i\}$ and $\Psi^{F_i}(x_i) \halts = \Psi^{F_i}(y_i) \halts = 1$; and}
                    \item {$a^{q_{n-1}}(x_i) = (S,|\pi^{q_{n-1}}|)$ and $a^{q_{n-1}}(y_i) = (I,|\pi^{q_{n-1}}|)$ if $x_i <_{\pi^{q_{n-1}}} y_i$, or $a^{q_{n-1}}(x_i) = (I, |\pi^{q_{n-1}}|)$ and $a^{q_{n-1}}(y_i) = (S,|\pi^{q_{n-1}}|)$}
                \end{itemize}
            Find a condition $q_n \lep q_{n-1}$ for which there is a chain $F_n$ in $\Phi^{\pi^{q_n}}$ such that both $\Psi^{F_n}(x_n) \halts = \Psi^{F_n}(y_n) \halts = 1$ for some pair $x_n,y_n < |\pi^{q_n}|$ with $\min\{x_n,y_n\} > \max\{x_i, y_i : i < n\}$, and $\min F_n \mid_{\Phi^{\pi^{q_n}}} \max F_i$ for each $i < n$. Additionally, ensure $x_n$ and $y_n$ exceed the stablization points of all $x_i, y_i$ for $i < n.$
            As before, finding $F_n$ depends only on $|\pi^{q_n}|$ and not $a^{q_n}$, so by moving to a parallel condition if necessary we can ensure $q_n$ additionally satisfies that $a^{q_n}(x_n) = (S,|\pi^{q_n}|)$ and $a^{q_n}(y_n) = (I,|\pi^{q_n}|)$ if $x_n <_{\pi^{q_n}} y_n$, or $a^{q_n}(x_n) = (I, |\pi^{q_n}|)$ and $a^{q_n}(y_n) = (S,|\pi^{q_n}|)$ otherwise.

            Such a $q_n$ and $F_n$ must exist and this is verified with an argument similar to the base case: if not, suppose $G$ is compatible with $q_{n-1}$ and satisfies the conclusion of Lemma \ref{nocompsoln}.
            Note that $\max F_i$ is isolated in $\Phi^G$ for each $i < n$ and $\Phi^G$ must contain an infinite chain.
            Thus, there is some infinite chain $X$ with $\min X \mid_{\Phi^G} \max F_i$ for all $i < n$.
            Furthermore, for every initial segment $F$ of $X$, there is no sufficiently large pair $x_n,y_n$ which in the set defined by $\Psi^F$.
            And so $\Psi^X$ defines a finite set, and we again contradict the assumption on $\Psi$.

            Now, set $E = \{\max F_i : i \le n\}$.
            Note $E$ is an antichain in $\Phi^{\pi^{q_n}}$ and $q_n$ forces every element of $E$ isolated in $\Phi^G$.
            Without loss of generality, assume $n$ and $|\pi^{q_n}|$ are sufficiently large so that there is a pair $x,y < |\pi^{q_n}|$ with $\Psi^E(x) \halts = \Psi^E(y) \halts = 1.$
            Again, if no such $n$ exists, then $E$ can be extended to a solution $X$ such that $\Psi^X$ is finite, a contradiction.
            
            The final step in our construction is to move to a condition $r$ which simultaneously diagonalizes $E$ and each set $F_0, \dots, F_n$.
            Let $D = \{x_i,y_i : i \le n\}$ be the set of witnesses for $F_0,\dots,F_n$ and note that we cannot guarantee that $x$ and $y$ are not already contained in $D$ or that $x$ and $y$ are beyond the stabilization points of the elements in $D$.
            Thus we may require a condition $r$ parallel to $q_n$ such that $a^r$ diagonalizes each of $E,F_0, \dots, F_n$ by making $G$ a poset of the large type. 

            There are three cases: either $x <_{\pi^{q_n}} y$, $y <_{\pi^{q_n}} x$, or $x \mid_{\pi^{q_n}} y$.
            The first two are symmetric so we assume without loss of generality that $x <_{\pi^{q_n}} y$ or $x \mid_{\pi^{q_n}} y$.
            
            \begin{quote}
                \noindent
                {\bf Case 1:} $x <_{\pi^{q_n}} y$.
                If possible, move to a condition $r$ parallel to $q_n$ such that $a^r(y) = (I, |\pi^{q_n}|)$ and for each $z \in D$ we have $a^r(z) = a^{q_n}(z)$.
                If not, then since $y$ cannot be made isolated while respecting the current limit behaviors assigned to the elements in $D$ there must be some $s \in D$ such that $y <_{\pi^{q_n}} s$ and $a^{q_n}(s) = (S,t_1)$ for some $t_1$.
                We claim in this case $x$ can be made small, i.e. there is a parallel condition $r'$ with $a^r(x) = (S, |\pi^{q_n}|)$ and $a^r(z) = a^{q_n}(z)$ for each $z \in D$.
                If not, then similarly there is some element $i \in D$ such that $i <_{\pi^{q_n}} x$ and $a^{q_n}(i) = (I,t_2)$ for some $t_2$.
                But then by transitivity $i <_{\pi^{q_n}} s$, so $a^{q_n}(i) = (S,u)$ for some $u$, a contradiction.
                So there is a condition $r'$ parallel to $q_n$ which makes both $x$ and $y$ small, since $y$ cannot be isolated, and agrees with $q_n$ on the limit behavior of the elements in $D$.
                Apply Lemma \ref{mindchange} to $r'$ to obtain a parallel condition $r$ which makes every small element in $\pi^{q_n}$ isolated and every isolated element in $\pi^{q_n}$ large.
                Clearly, $r$ diagonalizes $E$ as $a^r$ makes both $x$ and $y$ isolated.
                To see that $r$ also diagonalizes $F_0, \dots, F_n$, fix $F_i$ for some $i \le n$.
                If $a^{q_n}(x_i) = (S,t)$ and $a^{q_n}(y_i) = (I,u)$ for some $t$ and $u$ then $x_i <_{\pi^{q_n}} y_i$.
                Since $a^r$ makes $x_i$ isolated and $y_i$ large, $\Psi^{F_i}$ still contains two comparable elements one of which is isolated.
                If on the other hand $a^{q_n}(x_i) = (I,t)$ and $a^{q_n}(y_i) = (S,u)$ then $y_i <_{\pi^{q_n}} x_i$ or $x_i \mid_{\pi^{q_n}} y_i$.
                As $a^r$ makes $x_i$ large and $y_i$ isolated, $\Psi^{F_i}$ either contains both comparable and isolated elements in the first case or both incomparable and isolated elements in the latter case.
                Hence $r$ is as desired.

                \noindent
                {\bf Case 2:} $x \mid_{\pi^{q_n}} y$.
                Here we need to make either $x$ or $y$ small so suppose this is not possible while respecting the limit behaviors of the elements in $D$.
                Then there is a condition $r'$ parallel to $q_n$ such that $a^{r'}(x) = a^{r'}(y) = (I,|\pi^{q_n}|)$ and $a^{r'}(z) = a^{q_n}(z)$ for all $z \in D$.
                Apply Lemma \ref{mindchange} to $r'$ to obtain a condition $r$ in every small element in $\pi^{q_n}$ is made isolated and every isolated element in $\pi^{q_n}$ is made large.
                As in Case 1, $r$ is the desired condition.
                In particular, $a^r$ makes $x$ and $y$ both large, so $\Psi^E$ contains two incomparable large elements, and simultaneously diagonalizes each of $F_0, \dots, F_n$.
            \end{quote}

            To conclude, note that the condition $r$ was found computably:
            indeed, the search for each $F_i$ had to succeed and we can computably extend conditions to conduct this search.
            Since for any condition $p$, the number of parallel conditions is finite, finding the correct limit behaviors for $r$ was computable as well.
            Thus, by Lemma \ref{nocompsoln}, there is an infinite computable stable poset $G$ extending $r$ with no computable solution.
            So $\Phi^G$ is an $\omega$-ordered stable poset and must contain an infinite antichain $E \cup X$ with $E < X$.
            And by construction, $\Psi^{E \cup X}$ contains two elements $x$ and $y$ such that 
            if $x$ and $y$ are comparable in $G$, then one of these elements is isolated, and
            if $x$ and $y$ are incomparable in $G$ then one of these elements is small if $G$ is of the small type, or one is large if $G$ is of the large type.
            Either way, $\Psi^{E \cup X}$ is not an infinite chain or infinite antichain in $G$.
            This contradicts the initial assumption that $\Phi$ and $\Psi$ witness $\scac \swred \scaco$ and the proof is complete.
        \end{proof}

To obtain the final result that $\scac \not \wred \scaco$, we note that $G$ was found uniformly computably in the indices of $\Phi$ and $\Psi$.

\begin{lem}\label{sw-pushdown}
        Given two problems $\sf P$ and $\sf Q$, if for each pair of functionals $\Phi$ and $\Psi$, there is a computable instance of $\sf P$ uniformly computable in (the indices of) these functionals which witnesses $\sf P \not \swred Q$ (via $\Phi$ and $\Psi$), then $\sf P \not \wred Q$.
    \end{lem}

    \begin{proof}
        Suppose $\sf P$ and $\sf Q$ are as hypothesized and $f(i,j)$ is a computable function which outputs (the index of) the instance $X$ of $\sf P$ witnessing $\sf P \not \swred \sf Q$ via the functionals $\Phi_i$ and $\Phi_j$. 
        For sake of contradiction assume that $\Phi_m$ and $\Phi_n$ witness $\sf P \wred Q$.
        Define a computable function $g(k)$ such that with $m$ and $n$ fixed we have
            \[
                \Phi_{g(k)}^Y = \Phi_n^{\Phi_{f(m,k)} \join Y}.
            \]
        Via the relativized recursion theorem, find a fixed point $k$ such that for all $Y$
            \[
                \Phi_{g(k)}^Y = \Phi_k^Y.
            \]
        We claim that $\Phi_m$ and $\Phi_k$ contradict $f(m,k)$, i.e., $\Phi_m$ and $\Phi_k$ witness $\sf P \swred Q$ with $X = \Phi_{f(m,k)}$.
        To see this, note that $f$ guarantees a solution $\hat Y$ to $\Phi_m^{\Phi_{f(m,k)}} = \Phi_m^X$ such that $\Phi_k^{\hat Y}$ is {\it not} a solution to $X = \Phi_{f(m,k)}$.
        But by construction
            \[
                \Phi_k^{\hat Y} = \Phi_{g(k)}^{\hat Y} = \Phi_n^{\Phi_{f(m,k)} \join \hat Y} = \Phi_n^{X \join \hat Y}.
            \]
        Since we assumed $\Phi_n$ witnesses, along with $\Phi_m$, that $\sf P \wred Q$, we have that $\Phi_n^{X \join \hat Y} = \Phi_k^{\hat Y}$ {\it is} a solution to $X = \Phi_{f(m,k)}$, a contradiction.
        Whence, we conclude that $\sf P \not \wred Q$.
    \end{proof}

    \begin{cor}
        $\scac \not \wred \scaco$
    \end{cor}

\subsection{Separation over strong computable reducibility}

We now turn our attention to strong computable reducibility and work to prove $\scac \not \scred \scaco$.
In the previous section, we leveraged the differences in the global behavior of general stable posets compared to $\omega$-ordered stable posets.
To diagonalize the instance of $\scaco$, we utilized the fact that an instance of $\scac$ can be of the small or large type whereas any instance of $\scaco$ is necessarily a stable poset of the small type.
To separate the principles under strong computable reducibility, we will exploit differences in the local behavior of instances of $\scac$ and $\scaco$.
In particular, we will utilize a technique known as {\it tree labeling} while constructing the instance $G$ of $\scac$ to obtain control over the ordering of a pair $x < y$ in the computed instance of $\scaco$, $\Phi^G$. 
Unless $\Phi^G$ has a relatively simple solution we can avoid computations from, we will force 
$y \le_{\Phi^G} x$ in the $\omega$-ordered poset $\Phi^G$ to produce the needed contradiction. 

While this is the basic idea, intricate machinery is required to create this situation.
This is because in a computable reduction we need to diagonalize all possible forward functionals simultaneously.
Our previous argument relied on the ability to find fresh pairs $x$ and $y$ at each step of our construction that did not conflict with previous work.
In a computable reduction, we cannot guarantee that such pairs could be found because while $x$ and $y$ may be used to diagonalize some set $\Phi_0^{\pi}$, this same pair could repeatedly be required to diagonalize $\Phi_n^{\pi}$ for infinitely many indices $n$, preventing us from diagonalizing all possible functionals $\Phi_n, \Psi_n$ at some finite stage. 

Therefore the basic strategy is as follows: we construct a stable poset $G$ using the forcing notion $\ds P$. 
For each valid forward functional $\Phi$ (that sends stable posets $G$ to $\omega$-ordered stable posets $\Phi^G$) we build two sets $C_\Phi$ and $A_\Phi$ and aim to ensure one of these sets cannot compute a solution to $G$.
If this is impossible, we will obtain a suitable set $R$ to add to a family of sets $\mathcal{D}$ each of which (when joined with $G$) cannot compute a solution to $G$ due to genericity.
One such approach must succeed, for otherwise we will arrive at the situation in which we can force a pair $x < y$ to have $x \le_{\Phi^G} y$, contradicting that $\Phi^G$ is $\omega$-ordered.

The construction will handle one valid triple of functionals $(\Phi, \Psi_0, \Psi_1)$ at a time.
By valid, we mean that $\Phi$ is a forward functional which produces $\omega$-ordered stable posets when given a stable poset $G$, and both $\Psi_0$ and $\Psi_1$ compute solutions to $G$ when given solutions to $\Phi^G$.
By appealing to Lachlan's Disjunction (Proposition \ref{LD}), we will need to only succeed in diagonalizing one of $\Psi_0^{C_\Phi}$ or $\Psi_1^{A_\Phi}$.
We say that diagonalizing $\Psi_0^{C_\Phi}$ is ``making progress on the chain'' and that diagonalizing $\Psi_1^{A_\Phi}$ is ``making progress on the antichain.''

\begin{prop}[Lachlan's Disjunction]\label{LD}
    Given sets $C$ and $A$, let $\mathcal{P}$ be an arithmetical property of sets.
    If for any pair of functionals $(\Psi_0,\Psi_1)$, we have $\mathcal{P}(\Psi_0^C)$ or $\mathcal{P}(\Psi_1^A)$ then for some $X \in \{C,A\}$ we have $\mathcal{P}(\Psi^X)$ for all functionals $\Psi$.
\end{prop}

\begin{proof}
    The contrapositive is a tautology.
    Fix $C$ and $A$.
    If there are functionals $\Psi_0$ and $\Psi_1$ such that $P(\Psi_0^C)$ does not hold and $P(\Psi_1^A)$ does not hold, then there is a pair of functionals, namely $(\Psi_0,\Psi_1)$, such that both $P(\Psi_0^C)$ and $P(\Psi_1^A)$ fail to hold. 
\end{proof}

\subsubsection{Tree labeling}

The bulk of technical work will arise in finding appropriate finite extensions of initial segments of $C_\Phi$ and $A_\Phi$.
Call these initial segments $C$ and $A$.
To make progress on the chain, we will to need find a finite set $F$ such that $\Psi_0^{C \cup F}$ presents a diagonalization opportunity against being a solution to $G$. 
If this fails, we will attempt to make progress on the antichain and find a similar finite set $F$ for which $\Psi_1^{A \cup F}$ can be prevented from forming a solution to $G$.
We conduct these searches symmetrically and in turn, beginning first with the chain $C$, before repeating the search in an attempt to extend $A$.
The key to the proof will be to gain leverage on the global structure of $\Phi^G$ after each failed search.
If both searches fail, we will gain total control on the limit behavior of a particular set of elements in $\Phi^G$ and become able to trap $\Phi^G$ from being both stable and $\omega$-ordered as mentioned above.

To conduct each search, we utilize {\it tree labeling} and therby gain control over the limit behavior of elements in $\Phi^G$.
This framework was first introduced by Dzhafarov \cite{dzhafarov_strong_2016}, and subsequently applied in Dzhafarov et al.~ \cite{dzhafarov_ramseys_2017} and Nichols \cite{nichols_effective_2019}.
The rough idea is to collect all potential extensions of a Mathias condition $(E,I)$ that present a diagonalization opportunity against some functional $\Delta$.
This is done by organizing all finite sets $F \subseteq I$ which have $\Delta^{E \cup F}(w)\halts = 1$ for some sufficiently large $w$ that has yet to be committed in the construction.
The formal definition is given next before verifying key properties.

\begin{dfn}
    For strings, $\alpha, \beta \in \omega^{<\omega}$, we let $a^\#$ abbreviate $\alpha \restriction |\alpha| - 1$ and $\alpha * \beta$ denote the concatenation of $\alpha$ and $\beta$. 
    For any $x \in \omega$, we let $\alpha * x = \alpha * \langle x \rangle$. 
    Thus $(\alpha * x)^\# = \alpha.$
    If $T \subseteq \omega^{< \omega}$ is a tree, then for any $\alpha \in T$, we call the set $R = \{x \in \omega: \alpha * x \in T\}$ the {\it row below} $\alpha$.
\end{dfn}

\begin{dfn}
    The {\it extension tree} $T(E,I,\Delta,n)$ for a given a Mathias condition $(E,I)$, Turing functional $\Delta$, and $n \in \omega$ is defined as follows: $\lambda \in T(E,I,\Delta,n)$ and $\alpha \in T(E,I,\Delta,n)$ if $\alpha$ is strictly increasing, $\ran(\alpha) \subseteq I$, and
        \[
            (\forall F \subseteq \ran(\alpha^\#))(\forall w \ge n) (\Delta^{E \cup F}(w) \simeq 0).
        \]
\end{dfn}

Notice $T(E,I,\Delta,n)$ is clearly closed under prefixes and thus is a subtree of $I^{<\omega}$.
For our purposes, we only consider functionals $\Delta$ with range contained in $\{0,1\}$.  
The key properties of the extension tree are summarized in the following lemma (which appears as Lemma 3.2 of \cite{dzhafarov_ramseys_2017}).

\begin{lem}
    The tree $T = T(E,I,\Delta,n)$ has the following properties
        \begin{enumerate}
            \item {If $T$ has an infinite path $f$, then $I' = \ran(f) \subseteq I$ satisfies
                \[
                    (\forall F \subseteq I')(\forall w \ge n)(\Delta^{E \cup F}(w) \simeq 0).
                \]
                }
            \item {If $\alpha \in T$ is not terminal, then for all $x \in I$ such that $x > \ran(\alpha)$, $\alpha * x \in T$. 
            }
            \item {If $\alpha \in T$ is terminal, then there is a finite set $F \subseteq \alpha$ such that
                \[
                    (\exists w \ge n)(\Delta^{E \cup F}(w) = 1).
                \]}
            \item {There is some $w \ge n$ with $\Delta^E(w) = 1$ if and only if $T(E,I,\Delta,n) = \{\lambda\}$.}
        \end{enumerate}
\end{lem}

\begin{proof}
    \begin{enumerate}
        \item {Suppose not. Then there is a set $F \subseteq I'$ and $w$ witnessing $\Delta^{E\cup F}(w) \halts = 1$. Without loss of generality, assume $F$ is finite. Let $\alpha \prec f$ be such that $F \subseteq \ran(\alpha^\#)$. By construction, $\alpha \not \in T$, a contradiction.}
        \item {If $\alpha \in T$ is non-terminal, then there is some $\beta \in T$ such that $\beta^\# = \alpha$. So every $F \subseteq \ran(\alpha)$ is such that $\Delta^{E \cup F}(w) \simeq 0$ for all $w \ge n$. Thus for every $x \in I$ with $x > \ran(\alpha)$, $\alpha * x \in T$ because $\alpha * x$ is increasing and $(\alpha * x)^\# = \alpha$.}
        \item {If $\alpha \in T$ is terminal, then for all $x \in I$ with $x > \ran(\alpha)$, $(\alpha * x) \not \in T$. Thus, as $(\alpha * x)^\# = \alpha$, there are witnesses $F \subseteq \ran(\alpha)$ and $w \ge n$ with $\Delta^{F \cup N}(w) \halts = 1$.}
        \item {For the if direction, note if $T = \{\lambda\}$ then $\lambda$ is terminal in $T$ and the statement follows from item 3. For the only if direction, notice $\alpha \in T$ implies there are no finite sets $F \subseteq \ran(\alpha^\#)$ such that $\Delta^{E \cup F}(w) = 1$ for some $w \ge n$. But this occurs when $F = \emptyset$. Thus if $\emptyset \subseteq \ran(\alpha^\#)$ and $\alpha \ne \lambda$ then $\alpha \not \in T$. Consequently, $T = \{\lambda\}$.}
    \end{enumerate}
\end{proof}

\begin{dfn}
    Suppose an extension tree $T = T(E,I,\Delta,n)$ is well-founded.
    Beginning with the terminal nodes of $T$, we recursively define a function $\lb : T \to \omega \cup \{\infty\}$ assigning to each $\alpha \in T$ a {\it label} $\lb(\alpha)$.
    If $\alpha \in T$ is terminal, there is some $w \ge n$ and $F \subseteq \ran(\alpha)$ such that $\Delta^{E \cup F}(w) = 1$.
    Let $\lb(\alpha)$ be the least such witness $w$.
    If $\alpha \in T$ is not terminal, assume for recursion that $\lb(\alpha * x)$ is defined for all $x \in I$ with $x > \ran(\alpha)$.
    If there is a number $w$ such that $\lb(\alpha * x) = w$ for infinitely many $x$, let $\lb(\alpha)$ be the least such $w$.
    Else, let $\lb(\alpha) = \infty$.
    We call $\lb$ a {\it labeling} of $T$ and say $\lb(\alpha)$ is {\it finite} if $\lb(\alpha) \in \omega$.
\end{dfn}

\begin{dfn}
    If $T = T(E,I, \Delta, n)$ is a well-founded extension tree with labeling $\lb$, we define the {\it labeled subtree} $T^L$ of $T$ recursively as follows.
    Place $\lambda \in T^L$.
    Now, assume $\alpha \in T^L$.
    If $\lb(\alpha)$ is finite, place $\alpha * x \in T^L$ for all $x$ such that $\lb(\alpha * x) = \lb(\alpha).$
    If $\lb(\alpha) = \infty$ and infinitely many immediate successors of $\alpha$ have label $\infty$, place each such successor into $T^L$.
    Else, $\lb(\alpha) = \infty$ and there are infinitely many finite labels $w$ such that $\lb(\alpha * x) = w$ for some $\alpha * x \in T$.
    In this case, for each such label $w$ place $\alpha * x$ into $T^L$ if $x$ is least such that $\lb(\alpha *x) = w$.
\end{dfn}

Note that if $\lb$ is a labeling of an extension tree $T$, then its restriction to the labeled subtree $T^L$ is a labeling of $T^L$.
		When the domain of $\lb$ is clear from context, we will use $\lb$ to refer to either the labeling of $T$ or $T^L$.

		The utility of this framework reveals itself when $\Delta$ is treated as a backward functional in a potential (strong) computable reduction.
		If $\lb(\lambda) = w$, the idea is to continue our construction in such a way so that $w$ will prevent $\Delta^X$ from being a solution to the original instance.
		Here $X$ is any infinite set compatible with the Mathias condition $(E,I)$.
		For example, if $\Delta^E$ is constructing a chain in our stable poset, we may make $w$ isolated to prevent $\Delta^X$ from being a solution.
		To force $w \in \Delta^X$ we search for a terminal string $\alpha \in T^L$.
		Since it must also have label $w$, there will be a set $F \subseteq \ran(\alpha)$ such that the Mathias condition $(E \cup F, \{x \in I : x > F\})$ will force $w \in \Delta^X$ for any compatible $X$.
		We then will have diagonalized $\Delta$ in the construction.
		The challenge will be finding $\alpha \in T^L$ such that $E \cup \ran(\alpha)$ has the desired structure in the original instance (e.g. $E \cup \ran(\alpha)$ is a chain of small elements).
		If our instance involves a stable limit behavior (e.g. being small or isolated), then we will be able to find some one point extension of a given $\alpha \in T^L$ as the row below $\alpha$ is infinite.
		At some point, the relationship of the elements in $\ran(\alpha)$ will have stabilized with respect to sufficiently large elements in the row below $\alpha$.
		So long as we find a non-terminal $\alpha \in T^L$ with a finite label, this approach will succeed.
		If instead, $\lb(\lambda) = \infty$ and our search through $T^L$ takes us to a pre-leaf $\alpha$ with $\lb(\alpha) = \infty$, then the labels of the successors $\alpha * x$ will all be distinct and finite.
		In this case, we may not be able to find a suitable $x$ and label $w$ to simultaneously extend our condition via $\ran(\alpha * x)$ while diagonalizing $\Delta$ with $w$.
		This will only occur if the limit behavior of $w$ directly determines the limit behavior of $x$, and in this way, we will have gained some control over the global structure of the computed instance in the reduction.
		Of course, at any point in this search, we may simply find an infinite set $R \subseteq I$ which will be suitable to add to $\mathcal{D}$.
  
\subsubsection{Proving the result}

We begin by proving a lemma which states that for any set $R$, a sufficiently generic poset $G$ resulting from our forcing notion $\ds P$ can avoid having solutions computable in $G \oplus R$.

\begin{lem}\label{suff}
    There is an $n$ such that for any set $R$ if $G$ is the poset resulting from a $\Sigma^0_n(R)$-generic filter in $\ds P$, then $G$ contains no $(G \oplus R)$-computable solution.
\end{lem}

\begin{proof}
    Fix a set $R$ and functional $\Gamma$.
    Let $W$ be the set of $\ds P$ conditions that force one of the following:
        \begin{itemize}
            \item {The set defined by $\Gamma^{G \oplus R}$ is finite.}
            \item {There are two elements $x,y \in \Gamma^{G \oplus R}$ such that $x$ is isolated in $P$ and $y$ is either small or large in $P$.}
        \end{itemize}
    Clearly, $W$ is uniformly arithmetic in $R$.
    So there is some $n$ such that $W$ is $\Sigma^0_n(R)$.
    It remains to show that $W$ is dense in $\ds P$.
    To see this, let $p$ be any condition in $\ds P$.
    Either there is a $q \le_{\ds P} p$ forcing $\Gamma^{G \oplus R}$ to define a finite set, in which case $q \in W$, or $p$ forces that $\Gamma^{G \oplus R}(x) \halts = 1$ for infinitely many $x$.
    In the latter case, there is a condition $r \le_{\ds P} p$ and two numbers $x,y > |\pi^p|$ such that $\Gamma^{\pi^r \oplus R}(x)\halts = \Gamma^{\pi^r \oplus R}(y)\halts = 1$.
    As the limit behavior of $x$ and $y$ does not effect the local structure of $\pi^r$, we can further assume without loss of generality that $a^r(x) = (I,t_0)$ and $a^r(y) = (S,t_1)$ for some $t_0$ and $t_1$.
    Thus $r \in W$ and the proof is complete.
\end{proof}


We are now ready to prove that there is a stable poset $G$ witnessing that $\scac \not \scred \scaco$.
To summarize our approach, for every triple of functionals $(\Phi,\Psi_0,\Psi_1)$ we seek to build an infinite chain $C_\Phi$ and an infinite chain $A_\Phi$ in $\Phi^G$ such that either $\Psi_0^{C_\Phi}$ is not a solution of $G$ or $\Psi_1^{A_\Phi}$ is not a solution of $G$.
We will first attempt to make progress on the chain $C_\Phi$, and then attempt to make progress on the antichain $A_\Phi$.
If both attempts fail, we may encounter a suitable infinite set $R$ to add to a collection $\mathcal{D}$.
At other stages of the construction, we will ensure $G$ is sufficiently generic over the elements of $\mathcal{D}$ so that $G$ has no $(G \oplus R)$-computable solutions for each $R \in \mathcal{D}$.
We will conclude by showing that if progress cannot be made on the chain or antichain, and no such $R$ can be found, then $\Phi^G$ is either not stable or not $\omega$-ordered.
The latter will be done by exemplifying a pair $a < b$ with $b \le_{\Phi^G} a$.
Note we use the notation $X < Y$ for sets $X$ and $Y$ to indicate $\max X < \min Y$ and the notation $Y < x$ to indicate $x > \max Y$.

\begin{thm}\label{potatoes}
            There is a stable poset $G$ and a collection of infinite sets $\mathcal{D}$ such that for any $R \in \mathcal{D}$, no $(G \oplus R)$-computable infinite set is a chain or antichain of $G$. Moreover, any $\omega$-ordered stable poset $\hat{G}$ computable from $G$ has either, an infinite chain or infinite antichain computable in $(G \oplus R)$ for some $R \in \mathcal{D}$, or an infinite chain or infinite antichain that computes no infinite chain or infinite antichain in $G$.
        \end{thm}

        \begin{proof}
            To construct the required objects we build the following:
                \begin{itemize}
                    \item {a sequence $p_0 \ge_{\ds P} p_1 \ge_{\ds P} \cdots$ of $\ds P$-conditions;}
                    \item {two sequences of finite sets $C_{\Phi,0} \subseteq C_{\Phi,1} \subseteq \cdots$, and $A_{\Phi,0} \subseteq A_{\Phi,1} \subseteq \cdots$ for each functional $\Phi$;}
                    \item {a decreasing sequence of infinite sets $R_0 \supseteq R_1 \supseteq \cdots$ such that $C_{\Phi,s} < R_s$ and $A_{\Phi,s} < R_s$ for all $s$ and $\Phi$; and}
                    \item {an increasing sequence of finite families $D_0 \subseteq D_1 \subseteq \cdots$ of infinite subsets of $\omega$.}
                \end{itemize}
            In the end $G = \bigcup_{n \in \omega} \pi^{p_n}$, $C_\Phi = \bigcup_s C_{\Phi,s}$ and $A_\Phi = \bigcup_s A_{\Phi,s}$, for each functional $\Phi$ and $\mathcal{D} = \bigcup_s D_s$.
            We need to make $G$ sufficiently generic over $\mathcal{D}$ and if possible ensure $C_\Phi$ and $A_\Phi$ are an infinite chain and infinite antichain of $\Phi^G$ respectively, one of which computes no infinite chain or infinite antichain of $G$.
            Toward these ends, we satisfy the following requirements for each $s \in \omega$ and all Turing functionals $\Phi$, $\Psi_0$ and $\Psi_1$.

            \begin{quote}
                \begin{itemize}
                    \item[$\mathscr{G}_s$:] $G$ is sufficiently $R$-generic for every $R \in D_s$.
                    \item[$\mathscr{D}_{\Phi,\Psi_0, \Psi_1}$:] If $\Phi^G$ is an $\omega$-ordered stable poset, then either $\Phi^G$ has an infinite chain or infinite antichain computable in $(G \oplus R)$ for some $R \in \mathcal{D}$, or there is an infinite chain $C_\Phi$ and an infinite antichain $A_\Phi$ in $\Phi^G$ such that one of $\Psi_0^{C_\Phi}$ or $\Psi_1^{A_\Phi}$ is not an infinite chain or infinite antichain in $G$.
                \end{itemize}
            \end{quote}
            Distribute the stages of the construction in such a way so that each $\mathscr{G}$ requirement is addressed infinitely often and each $\mathscr{D}$ requirement is addressed once.
            By Lemma \ref{suff}, the $\mathscr{G}$ requirements will ensure that $G$ has no infinite $(G \oplus R)$-computable chain or antichain for each $R \in \mathcal{D}$.
            By Lachlan's Disjunction (Proposition \ref{LD}), the $\mathscr{D}$ requirements will ensure that for each $\Phi$, one of $C_\Phi$ or $A_\Phi$ diagonalizes $\Phi^G$ with respect to any backward functional $\Psi$.

            \bigskip

            {\it Construction.}
                To begin, let $p_0 \in \ds P$ be any condition such that $a^p : |\pi^p| \to \{S, I\} \times (\omega \restriction |\pi^p| + 1)$.
                So $G$ will be a stable poset of the small type.
                Let $D_0 = \emptyset$ and let $C_{\Phi,0} = A_{\Phi,0} = \emptyset$ and $R_{\Phi,0} = \omega$ for all functionals $\Phi$.
                Suppose we are at stage $s$ and given $p_s, C_{\Phi,s}, A_{\Phi,s}$ for all $\Phi$, and $D_s$.
                Assume inductively that if $C_{\Phi,s}$ or $A_{\Phi,s}$ is nonempty for some $\Phi$ then $p_s$ forces $\Phi^G$ is an $\omega$-ordered stable poset and that $C_{\Phi,s}$ and $A_{\Phi, s}$ are respectively a chain of small elements and an antichain of isolated elements in $\Phi^G$.
                At the conclusion of the stage, if any of $p_{s+1}$, $C_{\Phi,s+1}$, $A_{\Phi,s+1}$, $R_{s+1}$ and $D_{s+1}$ have not been explicitly defined, let them equal $p_{s}$, $C_{\Phi,s}$, $A_{\Phi,s}$, $R_{s}$ and $D_{s}$ respectively.

            \bigskip

            {\it $\mathscr{G}$ requirements.}
                These are satisfied in a systematic and straightforward way.
                Let $n$ be sufficiently large to satisfy Lemma \ref{suff}.
                Suppose $s > t$ is the $\langle \ell,m \rangle$th stage dedicated to $\mathscr{G}_t$.
                If $\ell > |D_s|$, do nothing.
                Else, let $R$ be the $\ell$th set in $D_s$ and $W$ be the $m$th $\Sigma^0_n(R)$ set.
                If $p$ has an extension $q \in W$ set $p_{s+1} = q$ and proceed to the next stage.
                Otherwise, do nothing.

            \bigskip
            {\it $\mathscr{D}$ requirements.}
                Suppose stage $s$ is dedicated to requirement $\mathscr{D}_{\Phi,\Psi_0,\Psi_1}$.
                We begin with a few initial extensions of $p_s$, $C_{\Phi,s}$, $A_{\Phi,s}$ and $R_s$ before attempting to make progress on the chain $C_{\Phi,s}$ with an extension tree.
                If this fails, we repeat this search with a second extension tree to make progress on $A_{\Phi,s}$.
                Throughout, if at any point we encounter an infinite set $R$ which will contain only finitely many small or finitely many isolated elements of $\Phi^G$, we add $R$ to $D_s$ and proceed to the next stage.
                If each of these approaches fail, we will contradict that $\Phi^G$ is $\omega$-ordered.

            \medskip
            \noindent
            {\bf Initialization.}

                Begin by extending $p_s$ if necessary to a condition forcing that $\Phi^G$ is an $\omega$-ordered stable poset.
                If no such extension exists, $\mathscr{D}_{\Phi,\Psi_0,\Psi_1}$ is vacuously satisfied and we proceed to the next stage.
                Next, search for an extension $q$ forcing either of the following statements:
                    \begin{itemize}
                        \item {The number of small elements in $\Phi^G \cap R_s$ is finite.}
                        \item {The number of isolated elements in $\Phi^G \cap R_s$ is finite.}
                    \end{itemize}
                If so, then there is set $D$ cofinite in $R_s$ such that $D$ will contain only small or only isolated elements of $\Phi^G$.
                Hence there is an $(R_s \oplus G)$-computable infinite chain or antichain of $\Phi^G$.
                So we set $D_{s + 1} = D_s \cup \{R_s\}$ and $p_{s + 1} = q$ and proceed to the next stage, having satisfied $\mathscr{D}_{\Phi,\Psi_0,\Psi_1}$.

                If there is no such condition, we conclude that $R_s$ contains infinitely many small elements and infinitely many isolated elements of $\Phi^G$.
                Move to an extension $q \le_{\ds P} p_s$ and a triple of sets $(C,A,I)$ such that
                    \begin{itemize}
                        \item {$ C \subseteq C_{\Phi,s} \cup R_s$ and $C$ is a finite chain in $\Phi^{\pi^q}$ with $|C_{\Phi,s}| < |C|$;}
                        \item {$ A \subseteq A_{\Phi,s} \cup R_s$ and $A$ is a finite antichain in $\Phi^{\pi^q}$ with $|A_{\Phi,s}| < |A|$;}
                        \item {$I$ is the infinite set $\{x \in R_s : x > A \land x > C\}$; and}
                        \item {$q$ forces that each $x \in C$ is small in $\Phi^G$, each $x \in A$ is isolated in $\Phi^G$, and for all $x > |\pi^q|$, $C \cup \{x\}$ and $A \cup \{x\}$ are respectively a chain and antichain in $\Phi^G$.}
                    \end{itemize}
                Note we have extended $C_{\Phi,s}$ and $A_{\Phi,s}$ by a finite amount, ensuring that $C_\Phi$ and $A_\Phi$ will be infinite.
                Additionally, choose $C$ and $A$ large enough such that $\Psi_0^C$ and $\Psi_1^A$ have size at least two.
                If this is not possible, then one of $\Psi_0^{C_\Phi}$ or $\Psi_1^{A_\Phi}$ must be finite.
                In this case, we set $C_{\Phi,s+1} = C$, $A_{\Phi,s+1} = A$ and $R_{s+1} = I$ and proceed to the next stage having made progress on at least one of the chain or the antichain and thereby satisfied $\mathscr{D}_{\Phi,\Psi_0,\Psi_1}$.

                Otherwise, we seek to extend $p_s$, $C$, $A$ and $I$ to $p_{s+1}$, $C_{\Phi,s+1}$, $A_{\Phi,s+1}$ and $R_{s+1}$ via tree labeling.
                We first attempt to make progress on $C$ with an extension tree for $\Psi_0$.
                We then attempt to make progress on $A$ with an extension tree for $\Psi_1$.
                Throughout both attempts, if at any point an infinite set $R$ is found which will only contain finitely many small or finitely many isolated elements of $\Phi^G$ we will proceed to the next stage after setting $D_{s + 1} = D_s \cup \{R\}$.
                Unless explicitly defined otherwise, we let $p_{s+1} = p_s$, $C_{\Phi,s+1} = C$, $A_{\Phi,s+1} = A$, and $R_{s+1} = I$.

            \medskip
            \noindent
            {\bf Making progress on the chain.}

                Construct the extension tree $T(C,I,\Psi_0, |\pi^{q}|)$.
                If this tree has a path $f$, let $R_{s+1} = \ran(f)$ and proceed to the next stage.
                Note this satisfies $\mathscr{D}_{\Phi,\Psi_0,\Psi_1}$ as the set defined by $\Psi_0^{C_\Phi}$ will be finite.
                Indeed, for sufficiently large $n$ the oracle use of $\Psi_0^{C_\Phi}(n)$ is contained in $C\cup F$ for some finite $F \subseteq \ran(f)$ and the definition of the extension tree ensures $\Psi_0^{C \cup F}(n) \simeq 0$.

                If $T(C,I,\Psi_0,|\pi^{q}|)$ is well-founded, construct a labeling $\lb : T \to \omega \cup \{\infty\}$ and the corresponding labeled subtree $T^L_C$.
                We aim to determine a finite sequence of $\ds P$ conditions $q \ge_{\ds P} q_0 \ge_{\ds P} q_1 \ge_{\ds P} \dots \ge_{\ds P} q_n \ge_{\ds P} p_{s+1}$ alongside a sequence of strings $\alpha_0 \preceq \alpha_1 \preceq \dots \preceq \alpha_n \preceq \alpha$ in $T^L_C$ in which $\alpha$ is terminal and has an appropriate $F \subseteq \ran(\alpha_n)$ to extend $C$ to $C_{\Phi,s+1}$.

                Let $\alpha_0 = \lambda$. 
                If $\lb(\alpha_0) = w$, find an extension $q_0 \le_{\ds P} q$ such that $a^{q_0}(w) = (I,|\pi^{q_0}|)$ if $\Psi_0^C$ contains two comparable elements in $\pi^q$ or $a^{q_0}(w) = (S,|\pi^{q_0}|)$ otherwise.
                If $\lb(\alpha_0) = \infty$, set $q_0 = q$.
                Assume we have obtained a condition $q_n$ and nonterminal string $\alpha_n \in T^L_C$ such that $\ran(\alpha_n)$ is a chain in $\Phi^{\pi^{q_n}}$ and $q_n$ forces each $x \in \ran(\alpha_n)$ small in $\Phi^G$.
                To extend $\alpha_n$, we consider two cases depending on whether or not the label of $\alpha_n$ agrees with the label of each of its successors.

                Case 1: $\lb(\alpha_n) = \lb(\alpha_n * x)$ for each $x$ with $\alpha * x \in T^L_C$.
                As each element of $\ran(\alpha_n)$ is forced small, there will be some $t$ such that for all $x > t$ with $\alpha_n * x \in T^L_C$, the set $\ran(\alpha_n) \cup \{x\}$ will be a chain in $\Phi^G$.
                Search for a condition $q_{n + 1} \le_{\ds P} q_n$ and $x$ such that $\alpha_n * x \in T^L_C$, $\ran(\alpha_n * x)$ is a chain in $\Phi^{\pi^{q_{n + 1}}}$ and $q_{n + 1}$ forces $x$ small in $\Phi^G$.
                Let $\alpha_{m + 1} = \alpha_n * x$ and proceed.
                If there is no such condition $q_{n + 1}$, then $q_n$ forces that the row below $\alpha_n$, $R = \{x : \alpha_n * x \in T^L_C\}$, will contain only finitely small elements of $\Phi^G$.
                So add $R$ to $D_s$, set $p_{s + 1} = q_n$ and proceed to the next stage having satisfied the requirement.

                Case 2: $\lb(\alpha_n) \ne \lb(\alpha_n * x)$ for some $\alpha_n * x \in T^L_C$.
                Notice this case can only occur if $\lb(\alpha_n) = \infty$ and each of its immediate successors $\alpha_n * x$ have distinct finite labels.
                Note again that any condition $q_{n+1} \le_{\ds P} q_n$ will guarantee $\ran(\alpha_n * x)$ is a chain in $\Phi^G$ for sufficiently large $x$.
                Search for a condition $q_{n + 1} \le_{\ds P} q_n$, and an $x$ and $w$ such that $\alpha_n * x \in T^L_C$, $\lb(\alpha_n * x) = w$, $\ran(\alpha_n * x)$ is a chain in $\Phi^{\pi^{q_{n+1}}}$, and that either
                    \begin{itemize}
                        \item {$q_{n+1}$ forces $x$ small in $\Phi^G$ with $a^{q_{n+1}}(w) = (I,|\pi^{q_{n+1}}|)$ if $\Psi_0^C$ has two comparable elements in $\pi^q$; or}
                        \item {$q_{n+1}$ forces $x$ small in $\Phi^G$ with $a^{q_{n+1}}(w) = (S,|\pi^{q_{n+1}}|)$ if $\Psi_0^C$ has two incomparable elements in $\pi^q$.}
                    \end{itemize}
                Let $\alpha_{n + 1} = \alpha_n * x$ and proceed.
                If no such condition exists, let
                    \[
                        R_C = \{x : \alpha_n * x \in T^L_C \land \ran(\alpha_n * x) \mbox{ will be a chain in } \Phi^G \land \lb(\alpha * x) > |\pi^{q_n}|\}
                    \] 
                and set $R^L_C = \{\langle x, w \rangle : x \in R_C \land \lb(\alpha_n * x) = w\}$.
                Note $R_C$ is an infinite subset of the row below $\alpha_n$.
                For every $\langle x,w \rangle \in R^L_C$, if $q' \le_{\ds P} q_n$ and $a^{q'}$ assigns the required limit behavior to $w$ (depending on whether $\Psi_0^C$ has two comparable or incomparable elements in $\pi^q$) then $q'$ cannot force $x$ small in $\Phi^G$.
                Hence, $q'$ must force $x$ isolated in $\Phi^G$.
                In this case, we have failed to extend $\alpha_n$ but have gained some control on the limit behavior of $x \in R_C$.
                So we must attempt to make progress on the antichain.

                If instead we successfully extend $\alpha_0 \le \dots \le \alpha_n$ to a terminal string $\alpha \in T^L_A$ at condition $q_n$ then $\lb(\alpha)$ is finite, say with value $w$.
                In this case there is some $F \subseteq \ran(\alpha)$ such that $\Psi_0^{C \cup F}(w)\halts = 1$, and $q_n$ forces that every element in the chain $C \cup F$ is small in $\Phi^G$.
                Moreover, $q_n$ will make $w$ isolated in $P$ if $\Psi_0^{C \cup F}$ contains two comparable elements of $\pi^q$, and otherwise make $w$ small.
                Either way, any infinite chain $X$ extending $C \cup F$ in $\Phi^G$ will have that $\Psi_0^X$ is not an infinite chain or antichain of $G$.
                We set $p_{s+1} = q_n$, $C_{\Phi,s+1} = C \cup F$ and $R_{s+1} = \{x \in I : x > C \cup F \cup A\}$ and proceed to the next stage, having made progress on the chain.

                If at any point, we fail to extend some $\alpha_n$ then we either found a set $R$ to add to $D_s$ or we are at a condition $q_n$ with the sets $R_C$ and $R^L_C$.
                With newfound leverage, we repeat the search in an extension tree for the antichain $A$.
            
            \medskip
            \noindent
            {\bf Making progress on the antichain.}

                Construct the extension tree $T(A,R_C,\Psi_1,|\pi^{q_n}|)$.
                As with the previous extension tree, if $T(A,R_C,\Psi_1,|\pi^{q_n}|)$ has an infinite path $f$, set $p_{s+1} = q_n$, and $R_{s+1} = \ran(f)$ and proceed to the next stage having made progress on the antichain.
                Otherwise, $T(A,R_C,\Psi_1,|\pi^{q_n}|)$ is well-founded.
                Form a labeling $\lb$ of $T(A,R_C,\Psi_1,|\pi^{q_n}|)$ and the corresponding labeled subtree $T^L_A$.
                We again aim to construct a finite sequence of conditions $q_n \ge_{\ds P} q_{n+1} \ge_{\ds P} \dots \ge_{\ds P} q_m \ge_{\ds P} p_{s+1}$ and a terminal string $\alpha \in T^L_A$ containing a suitable $F \subseteq \ran(\alpha)$ with which to extend $A$ to $A_{\Phi,s+1}$.

                We build a sequence $\alpha_1 \preceq \alpha_2 \preceq \alpha_m \preceq \alpha \in T^L_A$ similar to above while highlighting the differences.
                Let $\alpha_1 = \lambda$. If $\lb(\alpha_1) = w$, move to an extension $q_{n + 1} \le_{\ds P} q_n$ with $a^{q_{n+1}}(w) = (I,|\pi^{q_{n+1}}|)$ if $\Psi_1^A$ contains two comparable elements in $\pi^q$ or $a^{q_{n+1}}(w) = (S,|\pi^{q_{n+1}}|)$ otherwise.
                If $\lb(\alpha_1) = \infty$, set $q_{n+1} = q_n$.
                Assume we have constructed $\alpha_m$ and found condition $q_{n + m}$ such that $\ran(\alpha_m)$ is an antichain in $\Phi^{\pi^{q_{n + m}}}$ and $q_{n + m}$ forces each element in $\ran(\alpha_m)$ isolated in $\Phi^G$.
                We again consider two cases to extend $\alpha_m$ by one element.

                Case 1: $\lb(\alpha_m) = \lb(\alpha_m * x)$ for all $x$ with $\alpha_m * x \in T^L_A$.
                Since the elements of $\ran(\alpha_m)$ are forced isolated in $\Phi^G$, we conclude $\ran(\alpha * x)$ will be an antichain in $\Phi^G$ for sufficiently large $x$ with $\alpha * x \in T^L_A$.
                Thus, if there is an $x$ and condition $q_{n+m+1}$ such that $\alpha * x \in T^L_A$, $\ran(\alpha * x)$ is an antichain in $\Phi^{\pi^{q_{n+m+1}}}$, and $q_{n+m+1}$ forces $x$ isolated in $\Phi^G$, let $\alpha_{m + 1} = \alpha_m * x$ and proceed.
                Otherwise, $q_{n + m}$ forces the row $R$ below $\alpha_m$ in $T^L_A$ to contain only finitely many isolated elements of $\Phi^G$.
                So we set $p_s = q_{n+m}$, $R_{s+1} = R_C$, and $D_{s+1} = D_s \cup \{R\}$ and conclude this stage of the construction having satisfied the requirement.

                Case 2: $\lb(\alpha_m) \ne \lb(\alpha_m * x)$ for some $\alpha_m * x \in T^L_C$.
                Then $\lb(\alpha_m) = \infty$ and for sufficiently large $x$, $q_{n + m}$ forces $\ran(\alpha_m * x)$ to be an antichain in $\Phi^G$.
                Search for a condition $q_{n + m + 1} \le_{\ds P} q_{n+m}$, and an $x$ and $w$ such that $\alpha_m * x \in T^L_A$, $\lb(\alpha_m * x) = w$, $\ran(\alpha_m * x)$ is an antichain in $\Phi^{\pi^{q_{n+m+1}}}$, and that either
                    \begin{itemize}
                        \item {$q_{n+m+1}$ forces $x$ isolated in $\Phi^G$ with $a^{q_{n+m+1}}(w) = (I,|\pi^{q_{n+m+1}}|)$ if $\Psi_1^A$ has two comparable elements in $\pi^q$; or}
                        \item {$q_{n+m+1}$ forces $x$ isolated in $\Phi^G$ with $a^{q_{n+m+1}}(w) = (S,|\pi^{q_{n+m+1}}|)$ if $\Psi_1^A$  has two incomparable elements in $\pi^q$.}
                    \end{itemize}
                Let $\alpha_{m + 1} = \alpha_m * x$ and proceed.
                If no such condition exists, let
                    \[
                        R_A= \{x : \alpha_m * x \in T^L_C \land \ran(\alpha_m * x) \mbox{ will be an antichain in } \Phi^G \land \lb(\alpha * x) > |\pi^{q_{n+m}}|\}.
                    \] 
                and set $R^L_A= \{\langle x,w \rangle : x \in R_A \land \lb(\alpha * x) = w \}$.
                As before $R_A$ is an infinite subset of the row below $\alpha_m$.
                And similar to $R^L_C$, we have that for any $\langle x,w \rangle \in R^L_A$, if $q' \le_{\ds P} q_{n + m}$ and assigns the correct limit behavior to $w$, then $q'$ forces $x$ small in $\Phi^G$.
                In this situation, we will make progress on the antichain by finding a finite set $F \subseteq R_A$ to extend $A$.

                If we successfully extend $\alpha_1 \preceq \dots \preceq \alpha_m$ to a terminal string $\alpha \in T^L_A$ at condition $q_{n+m}$, then $\lb(\alpha) = w$ for some $w$.
                There is a set $F \subseteq \ran(\alpha)$ such that $\Psi_1^{A \cup F}(w) \halts = 1$.
                Furthermore, $q_{n + m}$ forces $A \cup F$ to be an antichain of isolated elements in $\Phi^G$ and assigns $w$ the correct limit behavior to ensure any infinite antichain $X$ extending $A \cup F$ does not define via $\Psi_1$ an infinite chain or antichain of $G$.
                Thus we have made progress on the antichain and set $p_{s+1} = q_{n + m}$, $A_{\Phi,s+1} = A \cup F$, $R_{s+1} = \{x \in R_C : x > A \cup F \cup C\}$ before proceeding to the next stage.

                If we failed to extend some $\alpha_m$, then we either found a set $R$ to add to $D_s$ or we are at condition $q_{n+m}$ with the sets $R_A$ and $R^L_A$ as well as the sets $R_C$ and $R^L_C$.
                Search for a finite set $F \subseteq R_A$ such that $\Psi_1^{A \cup F}(a)\halts = \Psi_1^{A \cup F}(b)\halts = 1$ for two fresh elements $a,b > |\pi^{q_{n+m}}|$. 
                We claim there is an extension $q' \le_{\ds P} q_{n+m}$ which forces $A \cup F$ to be an antichain of isolated elements in $\Phi^G$ and guarantees that $\Psi_1^{A \cup F}$ contains two incomparable but small elements or two comparable but isolated elements of $P$.
                Thus we set $p_{s+1} = q'$, $A_{\Phi,s+1} = A \cup F$, $R_{s+1} = \{x \in R_A : x > A \cup F \cup C\}$ and proceed to the next stage having made progress on the antichain.
                If there is no such set $F \subseteq R_A$, then set $R_{s+1} = R_A$ and $p_{s+1} = q_{n+m}$.
                This guarantees that any infinite antichain $X$ extending $A$ will have $\Psi_1^X$ finite so we make progress on the antichain in this case as well.
                This concludes the construction.

                \bigskip

                It remains to prove the claim and verify the construction.
                We first prove the claim.
                The idea is to show that if such a condition $q'$ does not exist then $\Phi^G$ is not $\omega$-ordered.

            \begin{sclaim}
                There is a condition $q' \le_{\ds P} q_{n+m}$ which forces that $A \cup F$ is an antichain of isolated elements in $\Phi^G$ and there are two elements $a,b \in \Psi_1^{A \cup F}$ such that either 
                    \begin{enumerate}
                        \item {$\{a,b\}$ is an antichain in $\pi^{q'}$ and $a^{q'}$ sets both $a$ and $b$ to be small, or}
                        \item {$\{a,b\}$ is a chain in $\pi^{q'}$ and $a^{q'}$ sets both $a$ and $b$ to be isolated.}
                    \end{enumerate}
            \end{sclaim}

            \noindent
            {\it Proof of the claim.}
            Recall any $q \le_{\ds P} q_{n+m}$ forces that $A$ contains only isolated elements of $\Phi^G$ and for any $x > |\pi^q|$, that $A \cup \{x\}$ is an antichain in $\Phi^G$.
            So it suffices to show the existence of a condition $q' \le q_{n+m}$ which makes $F$ an antichain of isolated elements in $\Phi^G$ and for which one of statements 1 or 2 hold.

            To begin, note that $R_A \subseteq R_C$ by construction and moreover, $R^L_A \subseteq R^L_C$.
            If not, then there is an $x \in R_A$ such that $(x,w) \in R^L_A$ and $(x,w') \in R^L_C$.
            Extending $q_{n+m}$ to a condition in which $w$ and $w'$ have the correct limit behaviors forces $x$ to be isolated by its membership in $R_C$ and small by its membership in $R^A$, a contradiction.
            Thus we have complete control over the limit behavior of the elements of $F$ as $F \subseteq R_A \subseteq R_C$.
            To elaborate, for each $x \in F$ there is a $w$ such that $\langle x,w\rangle \in R^L_A$.
            Any condition $r \le_{\ds P} q_{n+m}$ determines whether $x$ will be small or isolated in $\Phi^G$ by the limit behavior it assigns $w$ (depending on the structure of $\Psi_0^C$ and $\Psi_1^A$).
            So we may ensure the elements of $F$ are isolated via the labels of its elements.
            Unless otherwise mentioned, we now only consider labels $w$ corresponding to $x \in F$.

            The operative fact here is that for any condition $r$, we may freely adjust the limit behaviors as needed of $a,b$ and the labels $w$. 
            This is because the limit behaviors $a^r$ assigns to $a,b$ and each $w$ has no effect on the local structure of $\pi^r$.
            So we may take a parallel condition $r'$ that assigns limit behavior in the way we need.
            This will not affect the local structure of $F$ as $\Phi^{\pi^r} = \Phi^{\pi^{r'}}$.
            Thus, if no $q'$ exists, we can conclude that this failure is not witnessed by limit behavior but instead locally.
            For example, if statement 1 fails this would mean that if $\{a,b\}$ is an antichain in $P$, then $F$ cannot itself be an antichain in $\Phi^G$.
            
            We will now show by contradiction that either statement 1 or 2 must hold.
            Since the way each label $w$ affects the corresponding element in $F$ depends on the structure of $\Psi_0^C$ and $\Psi_1^A$, we have four cases to consider.
            The concern will be when $a$ and $b$ are themselves labels of $F$.

            \begin{quote}
                {\bf Case 1:} $\Psi_0^C$ and $\Psi_1^A$ both have two comparable elements in $\pi^q$.
                Here if $\langle x,w \rangle \in R^L_A$, then any $r \le q_{n+m}$ with $w$ isolated forces $x$ to be both small and isolated in $\Phi^G$.
                So case 1 cannot obtain.

                \medskip 

                {\bf Case 2:} $\Psi_0^C$ and $\Psi_1^A$ both lack two comparable elements in $\pi^q$.
                Here if $\langle x,w \rangle \in R^L_A$, then any $r \le_{\ds P} q_{n+m}$ with $w$ small forces $x$ to be both small and isolated in $\Phi^G$.
                So case 2 cannot obtain.

                \medskip

                {\bf Case 3:} $\Psi_1^C$ lacks two comparable elements in $\pi^q$ but $\Psi_0^A$ does not.
                Here if $\langle x,w \rangle \in R^L_A$, then any $r \le_{\ds P} q_{n+m}$ which has $w$ isolated forces $x$ small in $\Phi^G$ because $x \in R_A$ and any $r \le_{\ds P} q_{n+m}$ with $w$ small forces $x$ isolated in $\Phi^G$ because $x \in R_C$.

                In this case, we show statement 1 must hold.
                Suppose not: that is, that there is no condition $q'$ which forces $F$ to be an antichain of isolated elements in $\Phi^G$ while having $\{a,b\}$ an antichain of small elements in $\pi^{q'}$.
                As mentioned above, this cannot occur due to limit behavior, but instead happens due to local behavior.
                Specifically, any condition $r \le_{\ds P} q_{n+m}$ which has $\{a,b\}$ an antichain in $\pi^r$ forces $x <_{\Phi^G} y$ for some $x,y \in F$.
                Choose a condition $r$ which makes the label of $x$ small, forcing $x$ isolated, and makes the label of $y$ isolated, forcing $y$ small.
                Note that as $a \mid_{\pi^r} b$, there are no concerns over whether or not $a$ and $b$ label $x$ and $y$.
                As $r$ forces $x <_{\Phi^G} y$ with $y$ small, $x$ must be small in $\Phi^G$.
                Then $x$ will be both small and isolated, a contradiction

                \medskip

                {\bf Case 4:} $\Psi_0^C$ has two comparable elements in $\pi^q$ but $\Psi_1^A$ does not.
                Here if $\langle x,w \rangle \in R^L_A$, then any $r \le_{\ds P} q_{n+m}$ which has $w$ isolated forces $x$ isolated in $\Phi^G$ because $x \in R_C$ and any $r \le_{\ds P} q_{n+m}$ with $w$ small forces $x$ small in $\Phi^G$ because $x \in R_A$.

                Here we show statement 2 must hold.
                Suppose not: then any condition $r \le_{\ds P} q_m$ setting $a <_P b$ forces $x <_{\Phi^G} y$ for some $x,y \in F$.
                From this, we conclude $\langle x,a \rangle, \langle y,b \rangle \in R^L_A$.
                That is, $a$ and $b$ are the labels corresponding to $x$ and $y$.
                To see this, note first that if $\langle x,b \rangle,\langle y,a \rangle \in R^L_A$ then any condition setting $b$ isolated and $a$ small, would in turn force $x$ isolated and $y$ small.
                But then the fact that $x <_{\Phi^G} y$ would imply $x$ is also small in $\Phi^G$, a contradiction.
                Next suppose $b$ is not the label for $y$.
                We can then force $x$ isolated and $y$ small via their labels, without affecting the limit behavior of $b$, and obtain a similar contradiction.
                {\it Mutatis mutandis}, we have that $a$ is the label for $x$.

                We also note that any condition setting $b <_P a$ forces $y <_{\Phi^G} x$.
                This follows from the structure of $T^L_A$ as every element in $R_A$ corresponds to a distinct label.
                Since $b <_P a$ forces some $w <_{\Phi^G} z$ with $w,z \in F$, we have that $b$ is the label for $w$ and $a$ the label for $z$.
                So $w = y$ and $z = x$.

                Summarizing, setting $a <_P b$ or $b <_P a$ forces $x <_{\Phi^G} y$ or $y <_{\Phi^G} x$ respectively.
                But $\Phi^G$ is $\omega$-ordered, so only one of $x <_{\Phi^G} y$ or $y <_{\Phi^G} x$ is valid.
                Assume it is the former, and move to a condition $r \le q_{n+m}$ such that $b <_P a$.
                Then $r$ forces that $y <_{\Phi^G} x$ with $x < y$, contradicting that $\Phi^G$ is $\omega$-ordered. 
            \end{quote}

            In any case, we show there must be a condition $q'$ for which statement 1 or 2 holds.
            This completes the proof of the claim.

            \bigskip

            {\it Verification.}
                Let $G$ be the generic poset and $\mathcal{D}$ the collection of infinite sets resulting from the construction.
                If $\Phi$ is a functional such that $\Phi^G$ is an $\omega$-ordered stable poset then either it has a $(G \oplus R)$-computable infinite chain or infinite antichain $X$ for some $R \in \mathcal{D}$, or two infinite sets $C_\Phi$ and $A_\Phi$ were constructed.
                In the first case, as each $\mathscr{G}$-requirement was satisfied, we have by Lemma \ref{suff} that $\Psi^X$ is not an infinite chain or infinite antichain of $G$ for any functional $\Psi$.
                In the second case, $C_\Phi$ and $A_\Phi$ are respectively a chain and antichain in $\Phi^G$.
                Furthermore, for every pair of functionals $(\Psi_0, \Psi_1)$ one of $\Psi_0^{C_\Phi}$ or $\Psi_1^{A_\Phi}$ is guaranteed not to be an infinite chain or infinite antichain of $G$ by the $\mathscr{D}$-requirements.
                Applying Lachlan's Disjunction (Proposition \ref{LD}) yields that for some $X \in \{C_\Phi,A_\Phi\}$, $\Psi^X$ is not an infinite chain or infinite antichain of $G$ for any functional $\Psi$.
                Thus $G$ is the desired stable poset and $\mathcal{D}$ is the desired collection of infinite sets.
        \end{proof}

        From this, the main result of this section immediately follows.

        \begin{cor}\label{main}
            $\scac \not \scred \scaco$
        \end{cor} 

        \begin{proof}
            This is witnessed by any stable poset satisfying Theorem \ref{potatoes}.
        \end{proof}


\bibliographystyle{elsarticle-num}
\bibliography{omega_cac}

\end{document}